\documentclass{elsart}

\pdfoutput=1

\usepackage{cite}
\usepackage{amsfonts,amsmath,latexsym,amssymb}

\usepackage{graphics}
\usepackage{graphicx,subfigure}

\usepackage{color}

\newcommand{\Div}{\mbox{\rm Div\,}}
\newcommand{\divm}{\mbox{\rm div\,}}

\newcommand{\bu}{\mbox{\boldmath{$u$}}}

\newcommand{\bE}{\mbox{\boldmath{$E$}}}
\newcommand{\bv}{\mbox{\boldmath{$v$}}}
\newcommand{\bw}{\mbox{\boldmath{$w$}}}

\newcommand{\bx}{\mbox{\boldmath{$x$}}}

\newcommand{\fb}{\mbox{\boldmath{$f$}}}

\newcommand{\bsigma}{\mbox{\boldmath{$\sigma$}}}

\newcommand{\btau}{\mbox{\boldmath{$\tau$}}}

\newcommand{\bvarepsilon}{\mbox{\boldmath{$\varepsilon$}}}

\newcommand{\bnu}{\mbox{\boldmath{$\nu$}}}
\newcommand{\beeta}{\mbox{\boldmath{$\eta$}}}

\newcommand{\bzero}{\mbox{\boldmath{$0$}}}

\newcommand{\A}{\mathcal{A}}
\newcommand{\E}{\mathcal{E}}
\newcommand{\D}{\mathcal{D}}
\newcommand{\B}{\mathcal{B}}

\newcommand{\T}{\mathcal{T}}

\newcommand{\G}{\Gamma}
\newcommand{\Ge}{\mathcal{G}}
\newcommand{\Beta}{\mathbf{\beta}}

\newcommand{\RR}{{\if mm {\rm I}\mkern -3mu{\rm R}\else \leavevmode
\hbox{I}\kern -.17em\hbox{R} \fi}}

\newcommand{\bvar}{\mbox{\boldmath{$\varepsilon$}}}
\newcommand{\bM}{\mbox{\boldmath{$M$}}}
\newcommand{\bF}{\mbox{\boldmath{$F$}}}

\newtheorem{theorem}{Theorem}[section]
\newtheorem{problem}[theorem]{Problem}
\newtheorem{remark}[theorem]{Remark}
\newtheorem{corollary}[theorem]{Corollary}
\newtheorem{proposition}[theorem]{Proposition}
\newenvironment{proof}[1]{{{\bf Proof. #1}}~}{\hfill$\qed$\par}%{\nopagebreak\mbox{}\par\addvspace{4mm}}

\def\theequation{\thesection.\arabic{equation}}  %FOR SETTING EQ.~(1.1)

\begin{document}

\begin{frontmatter}

\title{Analysis of a dynamic viscoelastic-viscoplastic piezoelectric contact problem}

\author[l1]{M. Campo}
\author[l2]{J.R. Fern\'andez\corauthref{c1}}
\author[l3]{\'A. Rodr\'{\i}guez-Ar\'os}
\author[l3]{J.M. Rodr\'{\i}guez}

\address[l1]{Centro Universitario de la Defensa, Escuela Naval Militar\\ Plaza de Espa\~na s/n,
36920 Mar\'{\i}n, Spain}

\address[l2]{Departamento de Matem\'atica Aplicada I, Universidade de Vigo\\
Escola de Enxe\~ner\'{\i}a de Telecomunicaci\'on, Campus As Lagoas
Marcosende s/n, 36310 Vigo, Spain}

\address[l3]{Departamento de M\'etodos Matem\'aticos e Representaci\'on\\ Universidade de A
Coru\~na, A Coru\~na, Spain}

\corauth[c1]{Corresponding author: Phone: +34986818746. Fax:
+34986812116. E-mail address: jose.fernandez@uvigo.es}

\begin{abstract}
In this paper, we study, from both variational and numerical points
of view, a dynamic contact problem between a
viscoelastic-viscoplastic piezoelectric body and a deformable
obstacle. The contact is modelled using the classical normal
compliance contact condition. The variational formulation is written
as a nonlinear ordinary differential equation for the stress field,
a nonlinear hyperbolic variational equation for the displacement
field and a linear variational equation for the electric potential
field. An existence and uniqueness result is proved using Gronwall's
lemma, adequate auxiliary problems and fixed-point arguments. Then,
fully discrete approximations are introduced using an Euler scheme
and the finite element method, for which some a priori error
estimates are derived, leading to the linear convergence of the
algorithm under suitable additional regularity conditions. Finally,
some two-dimensional numerical simulations are presented to show the
accuracy of the algorithm and the behaviour of the solution.

 {\bf Keywords}. Viscoelasticity, viscoplasticity,
piezoelectricity, existence and uniqueness, a priori error
estimates, nu\-me\-ri\-cal simulations.
\end{abstract}

\end{frontmatter}

\section{Introduction}

Dynamic contact problems for viscoelastic materials have been
studied in numerous publications. For instance, we could refer the
papers \cite{BN11, Cocou02, CR00, CS06, IN02, JE99, K97, KS01, LL09,
MO, MOS10} where these problems were considered assuming different
friction laws and types of contact (deformable and rigid obstacles
or bilateral contact, Coulomb's friction law, slip dependent
friction, etc). Moreover, the numerical approximation of these
problems were also done, including effects as the adhesion, the
piezoelectricity or the damage (see, e.g., \cite{Ahn08, AS09, BBHJ,
BFT08,CFHS05, CFKSV, CF13, DP15, FS_M2AN11}).

These viscoelastic materials have been utilized in many engineering
applications since they can be customized to meet a desired
performance while maintain low cost. An important issue concerning
such materials is that they may exhibit time-dependent and inelastic
deformations. The viscoelastic strain component consists of a
recoverable-reversible part (elastic strain) and a
re\-co\-ve\-ra\-ble-di\-ssi\-pative deformation part (inelastic
strain). When an inelastic strain is assumed to depend only on the
magnitude of the stress or strain, the term plastic strain is used.
When the plastic deformation also changes with time, like in the
viscous component of the viscoelastic part, the term viscoplastic
strain is used. Therefore, a combined viscoelastic-viscoplastic
cons\-titutive relationships should be considered.  Recently, new
models coupling both viscoplastic and viscoelastic effects have been
proposed (see, for instance, \cite{ASM14, BNQ07, CMV15, CXWWS13,
GC83, HL10, KM09, MSA14, MDD11}).

Piezoelectricity is the ability of certain cristals, like the quartz
(also ceramics (BaTiO3, KNbO3, LiNbO3, PZT-5A, etc.) and even the
human mandible or the human bone), to produce a voltage when they
are subjected to mechanical stress. This effect is characterized by
the coupling between the mechanical and the electrical properties of
the material. We note that this kind of materials appears usually in
the industry as switches in radiotronics, electroacoustics or
measuring equipments. Since the first studies by Toupin (\cite{T60,
T63}) and Mindlin (\cite{M68, M69}) a large number of papers have
been published dealing with related models (see, for instance,
\cite{I90, BY, MS91} and the references cited therein). During the
last ten years, numerous contact problems involving this
piezoelectric effect have been studied from the variational and
numerical points of view (see, i.e., \cite{BFO08, BFT08, BS09,
MOS10, LL09, LW06}).

In this paper, the contact is assumed to be with a deformable
obstacle and so, the classical normal compliance contact condition
is used (\cite{KMS, MO}). Moreover, a viscoelastic-viscoplastic
material is considered including, for the sake of generality in the
modelling, piezoelectric effects. Both variational and numerical
analyses are then performed, providing the existence of a unique
weak solution to the continuous problem and an a priori error
analysis for the fully discrete approximations. Finally, some
numerical simulations are presented in two-dimensional examples to
demonstrate the accuracy of the approximation and the behaviour of
the solution.

The outline of this paper is as follows. In Section \ref{section2},
we describe the ma\-the\-matical problem and derive its variational
formulation. An existence and uniqueness result is proved in Section
\ref{section3}. Then, fully discrete approximations are introduced
in Section \ref{section4} by using the finite element method for the
spatial approximation and an Euler scheme for the discretization of
the time derivatives. An error estimate result is proved from which
the linear convergence is deduced under suitable regularity
assumptions. Finally, in Section \ref{section5} some two-dimensional
numerical examples are shown to demonstrate the accuracy of the
algorithm and the behaviour of the solution.

\section{Mechanical problem and its variational formulation}\label{section2}\setcounter{equation}{0}

Denote by $\mathbb{S}^d$, $d=1,2,3$, the space of second order
symmetric tensors on $\mathbb{R}^d$ and by ``$\cdot$'' and
$\|\cdot\|$ the inner product and the Euclidean norms on
$\mathbb{R}^d$ and $\mathbb{S}^d$.

Let $\Omega\subset \mathbb{R}^d$ denote a domain occupied by a
viscoelastic-viscoplastic piezoelectric body with a Lipschitz
boundary $\Gamma=\partial \Omega$ decomposed into three
measu\-ra\-ble parts $\Gamma_D$, $\Gamma_F$, $\Gamma_C$,
 on one hand, and on two measurable parts  $\Gamma_A$ and $\Gamma_B$, on the other hand,
 such that $\hbox{meas }(\Gamma_D)>0$, $\hbox{meas }(\Gamma_A)>0$, and $\Gamma_C\subseteq\Gamma_B$.  Let $[0,T]$,
$T>0$, be the time interval of interest. The body is being acted
upon by a volume force with density $\fb_0$, it is clamped on
$\Gamma_D$ and surface tractions with density $\fb_F$ act on
$\Gamma_F$. Moreover, an electrical potential is pres\-cribed on
$\Gamma_A$ and electric charges are applied on $\Gamma_B$. Finally,
we assume that the body may come in contact with a deformable
obstacle on the boundary part $\Gamma_C$, which is located, in its
reference configuration, at a distance $s$, measured along the
outward unit normal vector
$\bnu=(\nu_i)_{i=1}^d$.% (see Fig. \ref{f-fig1}).

Let $\bx\in \Omega$ and $t\in [0,T]$ be the spatial and time
variables, respectively. In order to simplify the writing, some
times we do not indicate the dependence of various functions and
unknowns on $\bx$ and $t$. Moreover, a dot above a variable
represents its first derivative with respect to the time variable
and two dots indicate derivative of the second order.

Let $\bu=(u_i)_{i=1}^d\in \mathbb{R}^d$, $\varphi \in \mathbb{R}$,
$\bsigma=(\sigma_{ij})_{i,j=1}^d\in\mathbb{S}^d$ and
$\bvarepsilon(\bu)=(\varepsilon_{ij}(\bu))_{i,j=1}^d\in\mathbb{S}^d$
denote the displacements, the electric potential, the stress tensor
and the linearized strain tensor, respectively. We recall that
$$\varepsilon_{ij}(\bu)=\frac{1}{2} \left(\frac{\partial u_i}{\partial x_j}+\frac{\partial
u_j}{\partial x_i}\right), \quad i,j=1,\ldots,d.$$ The body is
assumed to be made of a viscoelastic-viscoplastic piezoelectric
material and it satisfies the following constitutive law (see, for
instance, \cite{DL,I90}),
\begin{equation}\label{const}
\begin{array}{l}
\displaystyle
\bsigma(\bx,t)=\A\bvarepsilon(\dot{\bu}(\bx,t))+\B\bvarepsilon(\bu(\bx,t))+\int_0^t
\Ge (\bsigma(\bx,s),\bvarepsilon(\bu(\bx,s))) \, ds \\
\qquad \qquad -\E^*{\bf E}(\varphi(\bx,t)),\end{array}
\end{equation}
where $\A=(a_{ijkl})$ and $\B=(b_{ijkl})$ are the fourth-order
viscous and elastic tensors, respectively, $\Ge$ is a viscoplastic
function whose properties will be detailed later, ${\bf
E}(\varphi)=(E_i(\varphi))_{i=1}^d$ represents the electric field
defined by
$$E_i(\varphi)=-\frac{\partial \varphi}{\partial x_i},\quad i=1,\ldots,d,$$
and $\E^*=(e_{ijk}^*)_{i,j,k=1}^d$ denotes the transpose of the
third-order piezoelectric tensor $\E=(e_{ijk})_{i,j,k=1}^d$. We
recall that
$$ e_{ijk}^*=e_{kij}, \quad \hbox{for all} \quad i,j,k=1,\ldots,d.$$

Following \cite{BY} the following constitutive law is satisfied for
the electric potential,
$$\D=\E\bvarepsilon(\bu)+\Beta {\bf E}(\varphi),$$
where $\D=(D_i)_{i=1}^d$ is the electric displacement field and
$\Beta=(\beta_{ij})_{i,j=1}^d$ is the electric permittivity tensor.

We turn now to describe the boundary conditions.

On the boundary part $\Gamma_D$ we assume that the body is clamped and thus the displacement
field vanishes there (and so $\bu=\bzero$ on $\Gamma_D\times (0,T)$). Moreover, since
the density of traction forces $\fb_F$ is applied  on the boundary part $\Gamma_F$,
it follows that $\bsigma \bnu=\fb_F$ on $\Gamma_F\times (0,T).$

The contact is assumed with a deformable obstacle and so, the
well-known normal compliance contact condition is employed for its
modeling (see \cite{KMS, MO}); that is, the normal stress
$\sigma_\nu=\bsigma\bnu\cdot \bnu$ on $\Gamma_C$ is given by
\[
 -\sigma_\nu = p(u_\nu-s),
\]
where $u_\nu= \bu \cdot \bnu $ denotes the normal displacement, in
such a way that, when $u_\nu>s$, the difference $u_\nu-s$ represents
the interpenetration of the body's asperities into those of the
obstacle. The normal compliance function $p$ is prescribed and it
satisfies $p(r) = 0$ for $r\leq 0$, since then there is no contact.
As an example, one may consider
\begin{equation*}\label{p}
p(r)= c_p \,r_+,
\end{equation*}
where $c_p>0$ represents a deformability constant (that is, it
denotes the stiffness of the obstacle), and $r_+={\rm
max}\,\{0,r\}.$ Formally, the classical Signorini non\-pe\-ne\-tra\-tion
conditions are obtained in the limit $c_p \to \infty$. We also assume
that the contact is frictionless, i.e. the tangential component of
the stress field, denoted by $\bsigma_\tau=\bsigma\bnu-
\sigma_\nu\bnu$, vanishes on the contact surface.

Let $\Omega$ be subject to a prescribed electric potential  on
$\Gamma_A$ and to a density of surface electric charges $q_F$ on
$\Gamma_B$, that is,
$$\begin{array}{l}
\varphi=\varphi_A \quad \hbox{on} \quad \G_A\times (0,T),\\
\D\cdot \bnu=q_F \quad \hbox{on} \quad \G_B\times (0,T).
\end{array}$$
For the sake of simplicity, we assume that no electric potential is
imposed on the boundary $\Gamma_A$ (i.e. $\varphi_A=0$), and that
$q_F=0$ on $\G_C$; that is, the foundation is supposed to be
insulator. We note that it is straightforward to extend the results
presented below to more general situations by decomposing $\Gamma$
in a different way and by introducing some modifications in the
analyses shown in the following sections.

The mechanical problem of the dynamic deformation of a
viscoelastic-visco\-plas\-tic piezoelectric body in contact with a
deformable obstacle is then written as follows.

{\bf Problem P}. {\it Find a displacement field
$\bu:\overline{\Omega}\times [0,T] \rightarrow \mathbb{R}^d$, a
stress field $\bsigma:\overline{\Omega}\times [0,T]\rightarrow
\mathbb{S}^d$, an electric potential field
$\varphi:\overline{\Omega}\times (0,T)\rightarrow \mathbb{R}$ and an
electric displacement field $\D:\overline{\Omega}\times
(0,T)\rightarrow \mathbb{R}^d$
 such that,
\begin{eqnarray}
&&\bsigma(\bx,t)=\A\bvarepsilon(\dot\bu(\bx,t))+\B\bvarepsilon(\bu(\bx,t))+\int_0^t
\Ge (\bsigma(\bx,s),\bvarepsilon(\bu(\bx,s))) \,
ds\nonumber\\
&& \qquad \qquad -\E^*{\bf E}(\varphi(\bx,t)) \quad
\hbox{for a.e.} \quad \bx\in\Omega,\, t\in (0,T), \label{eq1}\\
&&\D=\E\bvarepsilon(\bu)+\Beta {\bf E}(\varphi) \quad \hbox{in} \quad \Omega \times (0,T), \label{eq1b}\\
&&\rho \ddot{\bu}-\Div \bsigma= \fb_0 \quad \hbox{in} \quad \Omega\times (0,T), \label{eq2}\\
&&\divm \D =q_0 \quad \hbox{in} \quad \Omega\times (0,T), \label{eq2b}\\
&&\bu=\bzero \quad \hbox{on} \quad  \Gamma_D\times (0,T), \label{eq3}\\
&&\bsigma \bnu=\fb_F \quad \hbox{on}\quad \Gamma_F\times (0,T), \label{eq4}\\
&&\bsigma_\tau=\bzero,\quad -\sigma_\nu=p(u_\nu-s) \quad \hbox{on}\quad \Gamma_C\times (0,T), \label{eq5}\\
&&\varphi=0 \quad \hbox{on} \quad \G_A\times (0,T),\label{eq8}\\
&&\D\cdot \bnu=q_F \quad \hbox{on} \quad \G_B\times (0,T),\label{eq9}\\
&&\bu(0)=\bu_0, \quad \dot{\bu}(0)=\bv_0 \quad \hbox{in}\quad
\Omega. \label{eq7}
\end{eqnarray}}
Here, $\rho>0$ is the density of the material (which is assumed
constant for simpli\-ci\-ty), and $\bu_0$ and $\bv_0$ represent
initial conditions for the displacement and velocity fields,
respectively. $\fb_0$ is the density of the body forces acting in
$\Omega$ and $q_0$ is the vo\-lu\-me density of free electric
charges. Moreover, $\Div$ and $\divm$ represent the divergence
operators for tensor and vector-valued functions, respectively.

In order to obtain the variational formulation of Problem P, let us
denote by $H=[L^2(\Omega)]^d$ and define the variational spaces $V$,
$W$ and $Q$ as follows,
$$\begin{array}{l}
V=\{\bw \in [H^1(\Omega)]^d\, ; \, \bw=\bzero \quad \hbox{on} \quad \G_D\},\\
W=\{\psi\in H^1(\Omega)\, ;\, \psi=0 \quad \hbox{on} \quad \G_A\},\\
Q=\{\btau=(\tau_{ij})_{i,j=1}^d\in [L^2(\Omega)]^{d\times d} \,\,\,
; \,\,\, \tau_{ij}=\tau_{ji}, \, \; i,j=1,\ldots,d\}.
\end{array}$$

\begin{remark}\label{remark2.2}
We could assume $\rho$ to be more general. To do that we should
follow a standard procedure (see for example \cite[p.105]{SHSbook}).
If we assume
\begin{equation}\label{hip_rho}
\rho\in L^\infty(\Omega),\ \rho(\bx)\ge\rho^*>0\ {\rm a.e.}\
\bx\in\Omega,
\end{equation}
where $\rho^*$ is a constant, we shall use a modified inner product
in $H$, given by
$$
((\bu,\bv))_H=\int_\Omega\rho\bu\cdot\bv\,d\bx\quad
\forall\bu,\bv\in H.
$$
Let $|||\cdot|||_H$ denote the associated norm, i.e.,
$$
|||\bu|||_H^2=((\bu,\bu))_H^{\frac{1}{2}}\quad \forall\bu\in H.
$$
By (\ref{hip_rho}), the norm $|||\cdot|||_H$ is equivalent to the
usual $L^2$-norm.
\end{remark}

We will make the following assumptions on the problem data.

The viscosity tensor
$\A(\bx)=(a_{ijkl}(\bx))_{i,j,k,l=1}^d:\btau\in\mathbb{S}^d\rightarrow \A(\bx)(\btau)\in\mathbb{S}^d$
satisfies:
\begin{equation}
\begin{array}{l}
\mathrm{(a)\ } a_{ijkl}=a_{klij}=a_{jikl} \quad \hbox{for} \quad i,j,k,l=1,\ldots, d.\\
\mathrm{(b)\ } a_{ijkl}\in L^\infty(\Omega) \quad \hbox{for} \quad i,j,k,l=1,\ldots, d.\\
\mathrm{(c)\ \mathrm{T}h\mathrm{ere}\ exists\ }m_{\mathcal{A}}>0\mathrm{\ such \ that \ }
 \mathcal{A}(\bx)\btau\cdot \btau\geq m_{\mathcal{A}}\,\|\btau\|^2   \\
{}\qquad \forall \,\btau\in \mathbb{S}^d,\mathrm{\ a.e.\ }\bx\in \Omega.
\end{array}
\label{hip1}
\end{equation}

The elastic tensor
$\B(\bx)=(b_{ijkl}(\bx))_{i,j,k,l=1}^d:\btau\in\mathbb{S}^d\rightarrow \B(\bx)(\btau)\in\mathbb{S}^d$
satisfies:
\begin{equation}
\begin{array}{l}
\mathrm{(a)\ } b_{ijkl}=b_{klij}=b_{jikl} \quad \hbox{for} \quad i,j,k,l=1,\ldots, d.\\
\mathrm{(b)\ } b_{ijkl}\in L^\infty(\Omega) \quad \hbox{for} \quad i,j,k,l=1,\ldots, d.\\
\mathrm{(c)\ \mathrm{T}h\mathrm{ere}\ exists\
}m_{\mathcal{B}}>0\mathrm{\ such \ that \ }
 \mathcal{B}(\bx)\btau\cdot \btau\geq m_{\mathcal{B}}\,\|\btau\|^2   \\
{}\qquad \forall \,\btau\in \mathbb{S}^d,\mathrm{\ a.e.\ }\bx\in
\Omega.
%\mathrm{(c)\ } \mathcal{B}(\bx)\btau\cdot \btau\geq 0 \quad \forall
%\,\btau\in \mathbb{S}^d,\mathrm{\ a.e.\ }\bx\in \Omega.
\end{array}
\label{hip2}
\end{equation}

The piezoelectric tensor
$\E(\bx)=(e_{ijk}(\bx))_{i,j,k=1}^d:\btau\in\mathbb{S}^d\rightarrow
\E(\bx)(\btau)\in\mathbb{R}^d$ satisfies:
\begin{equation}
\begin{array}{l}
\mathrm{(a)\ } e_{ijk}=e_{ikj} \quad \hbox{for} \quad i,j,k=1,\ldots, d.\\
\mathrm{(b)\ } e_{ijk}\in L^\infty(\Omega) \quad \hbox{for} \quad
i,j,k=1,\ldots, d.
\end{array}
\label{hip3}
\end{equation}

The permittivity tensor
$\Beta(\bx)=(\beta_{ij}(\bx))_{i,j=1}^d:\bw\in\mathbb{R}^d\rightarrow
\Beta(\bx)(\bw)\in\mathbb{R}^d$ verifies:
\begin{equation}
\begin{array}{l}
\mathrm{(a)\ } \beta_{ij}=\beta_{ji}\quad \hbox{for} \quad i,j=1,\ldots, d.\\
\mathrm{(b)\ } \beta_{ij}\in L^\infty(\Omega) \quad \hbox{for} \quad i,j=1,\ldots, d.\\
\mathrm{(c)\ \mathrm{T}h\mathrm{ere}\ exists\ }m_{\Beta}>0\mathrm{\
such \ that \ }
\Beta(\bx)\bw\cdot \bw\geq m_{\Beta}\,\|\bw\|^2   \\
{}\qquad \forall \,\bw\in \mathbb{R}^d,\mathrm{\ a.e.\ }\bx\in
\Omega.
\end{array}
\label{hip4}
\end{equation}

The normal compliance function $p : \Gamma_C\times\mathbb{R}\longrightarrow\mathbb{R}^+$ satisfies:
\begin{equation}
\left. \label{hip5}  \begin{array}{ll}
 {\rm (a)\ There\ exists\ }L_{p}>0\ {\rm such\ that}\\
 {} \qquad |p (\bx,r_1)-p
(\bx,r_2)|\leq L_p \,|r_1-r_2|
  \quad\forall\,r_1,r_2\in \mathbb{R},   {\rm\ a.e.\ } \bx \in \Gamma_C.\\
{\rm (b)\  The\ mapping\ }\bx\mapsto p (\bx,r)\
  {\rm is\ Lebesgue\ measurable\ on\ }\Gamma_C,\\
\qquad \forall r\in\mathbb{R}.\\
{\rm (c)\  } (p(\bx,r_1)-p(\bx,r_2))\cdot(r_1-r_2)\geq 0
  \quad\forall\,r_1,r_2\in\mathbb{R},{\rm\ a.e.\ }\bx\in\Gamma_C.\\
{\rm (d)\   The\ mapping\ }\bx\mapsto p (\bx,r)=0\quad {\rm for\ all\ } r\le0.
\end{array}\right.
\end{equation}

The viscoplastic function $\Ge : \Omega\times
\mathbb{S}^d\times\mathbb{S}^d\rightarrow
\Ge(\bx)(\btau,\bvarepsilon)\in\mathbb{S}^d$ satisfies:
\begin{equation}
\begin{array}{rl}
&\hspace*{-0.7cm} \mathrm{(a)\ \mathrm{T}h\mathrm{ere}\ exists\
}L_{\mathcal{G}} >0\mathrm{\ such \ that \ } \\
&\hspace*{0cm}\left| \Ge \left(
\bx,\bsigma_1,\bvarepsilon_{1}\right) -\Ge \left(
\bx,\bsigma_2,\bvarepsilon_{2} \right) \right| \leq L_{\Ge }\left(
\left| \bvarepsilon_{1}-\bvarepsilon_{2}\right| +\left| \bsigma_1-\bsigma_2\right| \right)\\
&\hspace*{0cm}  \mbox{ for all }\bvarepsilon_{1},\bvarepsilon
_{2},\bsigma_1,\bsigma_2\in \mathbb{S}^{d},\quad \mbox{
a.e. }\ \bx\in \Omega.\\
&\hspace*{-0.7cm}\mathrm{(b)\ } \mbox{The function }\;\bx\rightarrow
\Ge
\left( \bx,\bsigma,\bvarepsilon \right) \mbox{ is measurable.}  \\
&\hspace*{-0.7cm}\mathrm{(c)\ } \hbox{The mapping  }\bx\rightarrow
\Ge \left( \bx ,\mathbf{0},\mathbf{0}\right) \hbox{  belongs to  }
Q.
\end{array}
\label{hip6}
\end{equation}
The following regularity is assumed on the density of volume forces
and tractions:
\begin{equation}
\begin{array}{l}
\fb_0\in C([0,T];H),\quad \fb_F\in C([0,T]; [L^2(\G_F)]^d),\\
q_0\in C([0,T];L^2(\Omega)),\quad q_F\in C([0,T]; L^2(\G_B)).
\end{array}
\label{hip7}
\end{equation}
Finally, we assume that the initial displacement and velocity
satisfy
\begin{equation}
 \bu_0,\bv_0\in V.
  \label{hip8}
\end{equation}
\begin{remark}
We can replace (\ref{hip8}) by asking a less restrictive condition
$\bu_0\in V,\,\bv_0\in H$.
\end{remark}

Moreover, we denote by $V'$ the dual space of $V$. We identify $H$ with its dual and consider the Gelfand triple
$$
V\subset H\subset V'.
$$
We use the notation $<\cdot,\cdot>_{V'\times V}$ to denote the
duality product and, in particular, we have
$$
<\bv,\bu>_{V'\times V}=(\bv,\bu)_H\quad\forall \bu\in V,\ \bv\in H.
$$

Using Riesz' theorem, from (\ref{hip7}) we can define the elements
$\fb(t)\in V'$ and  $q(t)\in W$ given by
$$\begin{array}{l}
\displaystyle \langle\fb(t),\bw\rangle_{V'\times V}=\int_\Omega \fb_0(t)\cdot \bw \, d\bx+
 \int_{\G_F} \fb_F(t)\cdot \bw\, d\Gamma \quad \forall \bw\in V,\\
 \displaystyle (q(t),\psi)_W=\int_\Omega q_0(t) \psi \, d\bx+ \int_{\G_B} q_F(t) \psi\, d\Gamma
\quad \forall \psi\in W,
 \end{array}$$
and then $\fb\in C([0,T];V')$ and $q\in C([0,T];W)$. Now, let us define the contact functional $j:V\times V \rightarrow \mathbb{R}$ by
\begin{equation*}
 j(\bu,\bv) =\displaystyle \int_{\Gamma_C} p(u_\nu-s)\, v_\nu\,d\Gamma \quad \forall \bu,\bv\in V,
\end{equation*}
where we let $v_\nu=\bv\cdot \bnu$ for all $\bv\in V$. Moreover,
from properties (\ref{hip5}) let us conclude that
\begin{equation}\label{2.16bis}
 j(\bu,\bw)-j(\bv,\bw)\le C\|\bu-\bv\|_V\|\bw\|_V\ \forall\,\bu,\bv,\bw\in V.
\end{equation}

Plugging (\ref{eq1}) into (\ref{eq2}), (\ref{eq1b}) into
(\ref{eq2b}) and using the previous boundary conditions, applying a
Green's formula we derive the following variational formulation of
Problem P, written in terms of the velocity field
$\bv(t)=\dot{\bu}(t)$ and the electric potential $\varphi(t)$.

\begin{problem}\label{PV}
Find a velocity field $\bv:[0,T] \rightarrow V$, a stress field
$\bsigma:[0,T]\rightarrow Q$ and an electric potential field
$\varphi:[0,T]\rightarrow W$ such that $\bv(0)=\bv_0$ and for a.e.
$t\in (0,T)$ and for all $\bw\in V$ and $\psi\in W$,
\begin{eqnarray}
&&\hspace*{-0.6cm}
\bsigma(t)=\A\bvarepsilon(\dot\bu(t))+\B\bvarepsilon(\bu(t))+\int_0^t
\Ge (\bsigma(s),\bvarepsilon(\bu(s))) \,
ds-\E^*{\bf E}(\varphi(t)),\label{var1}\\
&&\hspace*{-0.6cm} \displaystyle \nonumber \langle  \rho
\dot{\bv}(t),\bw\rangle_{V'\times V}
+\left(\A\bvarepsilon({\bv}(t))+\B\bvarepsilon(\bu(t))+\int_0^t \Ge (\bsigma(s),\bvarepsilon(\bu(s))) \, ds,\bvarepsilon(\bw)\right)_Q\\
&&\qquad \qquad + \left(\E^*{\bf
E}(\varphi(t)),\bvarepsilon(\bw)\right)_Q + j(\bu(t),\bw) =\langle
\fb(t),\bw\rangle_{V'\times V}, \label{var2}\\
&& \hspace*{-0.6cm} \label{var3} (\Beta \nabla \varphi(t),\nabla
\psi)_H-(\E\bvarepsilon(\bu(t)),\nabla \psi)_H= (q(t),\psi)_W,
\end{eqnarray}
where the displacement field $\bu(t)$ is given by
\begin{equation}
\bu(t)=\int_0^t \bv(s)\, ds + \bu_0.
\end{equation}
\end{problem}

\section{An existence and uniqueness result}\label{section3}\setcounter{equation}{0}

%Thus, from the assumptions (\ref{hip1})--(\ref{hip8}) we find that $\fb\in C(0,T;V')$.

\begin{theorem}\label{teo_PV}
Assume (\ref{hip1})--(\ref{hip8}) hold. Then, there exists a unique solution $(\bu,\bsigma,\varphi)$ to Problem \ref{PV}. Moreover, the solution satisfies
\begin{align}
&\bu\in H^1(0,T;V)\cap C^1([0,T];H),\quad\ddot\bu\in L^2(0,T;V'),\label{reg_u}\\
&\bsigma\in L^2(0,T;Q),\ \Div\bsigma\in L^2(0,T;V'),\label{reg_sigma}\\
&\varphi\in C([0,T];W).\label{reg_phi}
\end{align}
\end{theorem}

The proof of Theorem \ref{teo_PV} will be carried in several steps.
First, let $\bM\in L^2(0,T;Q)$ and consider the auxiliary problem.

\begin{problem}\label{PVM}
Find a velocity field $\bv_M:[0,T]\to V$ and an electric field $\varphi_M:[0,T]\to W$ such that $\bv_M(0)=\bv_0$ and for a.e. $t\in (0,T)$ and for all $\bw\in V$ and $\psi\in W$,
\begin{align}
&<\rho\dot{\bv}_M(t),\bw>_{V'\times V}+\left(\A\bvar(\bv_M(t))+\B\bvar(\bu_M(t)),\bvar(\bw)\right)_Q\nonumber\\
&\qquad \qquad +(\E^*\nabla\varphi_M(t),\bvar(\bw))_Q+j(\bu_M(t),\bw) \nonumber\\%
&\qquad=<\fb(t),\bw>_{V'\times V}-(\bM(t),\bvar(\bw))_Q,\label{eq_vM}\\
&(\beta\nabla\varphi_M(t),\nabla\psi)_H-(\E\bvar(\bu_M(t)),\nabla\psi)_H=(q(t),\psi)_W,\label{eq_phiM}
\end{align}
where the displacement field $\bu_M(t)$ is given by
\begin{equation}\label{def_uM}
\bu_M(t)=\int_0^t\bv_M(s)\,ds+\bu_0.
\end{equation}
\end{problem}

\begin{theorem}\label{teo_PVM}
Assume (\ref{hip1})--(\ref{hip8}) hold. Then, there exists a unique solution $(\bv_M,\varphi_M)$ to Problem \ref{PVM}. Moreover, the following regularities hold:
\begin{align}
&\bu_M\in H^1(0,T;V)\cap C^1([0,T];H),\quad\ddot{\bu}_M\in L^2(0,T;V'),\label{reg_uM}\\
&\varphi_M\in C([0,T];W).\label{reg_phiM}
\end{align}
\end{theorem}

To show the proof of this theorem we have to proceed in several
steps as well. Let $\beeta\in L^2(0,T;V')$ be given and consider the
following additional auxiliary problem.

\begin{problem}\label{PVM_u_eta}
Find a velocity field $\bv_{M\eta}:[0,T]\to V$ such that
$\bv_{M\eta}(0)=\bv_0$ and for a.e. $t\in (0,T)$ and for all $\bw\in
V$,
\begin{equation}\label{eq_u_eta}
\begin{array}{l}
\hspace*{-1cm}<\rho \dot{\bv}_{M\eta}(t),\bw>_{V'\times
V}+(\A\bvar(\bv_{M\eta}(t)),\bvar(\bw))_Q\\
\qquad =<\fb(t)-\beeta(t)-\bM(t),\bw>_{V'\times V},
\end{array}
\end{equation}
where the displacement field $\bu_{M\eta}(t)$ is given by
\begin{equation}\label{3.9bis}
\bu_{M\eta}(t)=\int_0^t\bv_{M\eta}(s)\,ds+\bu_0.
\end{equation}
\end{problem}
\begin{remark}
Note that in the right hand-side of variational equation
(\ref{eq_u_eta}) we make {\em un abus de langage}, since actually we
are identifying $\bM(t)\in Q$ with the corresponding $\bM(t)\in V'$
such that $(\bM(t),\bvar(\bw))_Q=<\bM(t),\bw>_{V'\times V}$ for all
$\bw\in V$.
\end{remark}

Now, we can apply a result proved in \cite[p. 107]{SHSbook} which we
can reformulate here as follows.
\begin{proposition}\label{teo_PVM_u_eta}
Assume (\ref{hip1})--(\ref{hip8}) hold. Then, there exists a unique
solution to Problem \ref{PVM_u_eta} and it has the regularity
expressed in (\ref{reg_uM}).
\end{proposition}

\begin{remark}
The result in \cite{SHSbook} is used in the framework of the study
of a viscoelastic dynamic contact problem,  based itself on an
abstract result found in \cite[p. 140]{Barbu}.
\end{remark}

We now consider the auxiliary problem for the electric part.

\begin{problem}\label{PVM_phi_eta}
Find an electric field $\varphi_{M\eta}:[0,T]\to W$ such that for a.e. $t\in (0,T)$ and for all $\psi\in W$,
\begin{equation}\label{eq_phi_eta}
(\beta\nabla\varphi_{M\eta}(t),\nabla\psi)_H=(\E\bvar(\bu_{M\eta}(t)),\nabla\psi)_H+(q(t),\psi)_W.
\end{equation}
\end{problem}

We have the following result.

\begin{proposition}\label{teo_PVM_phi_eta}
Assume (\ref{hip1})--(\ref{hip8}) hold. Then, there exists a unique
solution to Problem \ref{PVM_phi_eta} and it has the regularity
expressed in (\ref{reg_phiM}).
\end{proposition}
\begin{proof}{}
We define a bilinear form $b(\cdot,\cdot):W\times W\to\RR$ such that
$$
b(\varphi,\psi)=(\beta\nabla\varphi,\nabla\psi)_H\quad
\forall\varphi,\psi\in W.
$$
We use (\ref{hip4}) to show that the bilinear form is continuous,
symmetric and coercive on $W$. Moreover, using  (\ref{eq_phi_eta})
and (\ref{hip3}), the Riesz' representation theorem allows us to
define an element $q_\eta:[0,T]\to W$ such that
$$
(q_\eta(t),\psi)_W=(\E\bvar(\bu_{M\eta}(t)),\nabla\psi)_H+(q(t),\psi)_W\quad
\forall \psi\in W.
$$
We apply the Lax-Milgram theorem to deduce that there exists a
unique element $\varphi_{M\eta}(t)$ such that
$$
b(\varphi_{M\eta}(t),\psi)=(q_\eta(t),\psi)_W\quad \forall\psi\in W.
$$
We conclude that $\varphi_{M\eta}(t)$ is the solution to variational
equation (\ref{eq_phi_eta}). Moreover, by using (\ref{hip7}) and the
regularity of $\bu_{M\eta}$ and $q$, we conclude straightforwardly
that $\varphi_{M\eta}\in C([0,T];W)$.
\end{proof}

Now, let $\Lambda\beeta(t)$ denote the element of $V'$ defined by
\begin{align}
&<\rho\Lambda\beeta(t),\bw>_{V'\times V}\nonumber\\%
&\qquad=(\B\bvar(\bu_{M\eta}(t)),\bvar(\bw))_Q%
+(\E^*\nabla\varphi_{M\eta}(t),\bvar(\bw))_Q+j(\bu_{M\eta}(t),\bw),\label{def_lambda_eta}
\end{align}
for all $\bw\in V$ and $t\in[0,T]$. We have the following result.

\begin{proposition}\label{teo_ptofijo_eta}
For $\beeta\in L^2(0,T;V')$ it follows that $\Lambda\beeta\in
C([0,T];V')$ and the operator $\Lambda:L^2(0,T;V')\to L^2(0,T;V')$
has a unique fixed point $\beeta^*$.
\end{proposition}
\begin{proof}{}
The continuity of $\Lambda\beeta$ is a straightforward consequence
of the continuity of $\varphi_{M\eta}$ and $\bu_{M\eta}$. Let now
$\beeta_1,\beeta_2\in L^2(0,T;V')$ and $t\in[0,T]$. We use the
shorter notation $\bu_i=\bu_{M\eta_i}$, $\bv_i=\bv_{M\eta_i}$,
$\varphi_i=\varphi_{M\eta_i}$, for $i=1,2$. Then, taking
$\beeta=\beeta_i$ for $i=1,2$ successively in (\ref{def_lambda_eta})
and subtracting the resulting equations, we have, for all $\bw\in V$
and  $t\in (0,T),$
\begin{eqnarray*}
&&\hspace*{-0.7cm}<\rho\Lambda\beeta_1(t)-\rho\Lambda\beeta_2(t),\bw>_{V'\times V}=\left(\B(\bvar(\bu_1(t))-\bvar(\bu_2(t))),\bvar(\bw)\right)_Q\\
&&\qquad\qquad
+\left(\E^*\nabla(\varphi_1(t)-\varphi_2(t)),\bvar(\bw)\right)_Q+j(\bu_1(t),\bw)-j(\bu_2(t),\bw).
\end{eqnarray*}
By using (\ref{hip2}), (\ref{hip3}) and (\ref{2.16bis}), we find that
\begin{equation}\label{eq_des0}
\|\Lambda\beeta_1(t)-\Lambda\beeta_2(t)\|_{V'}\le C(\|\bu_1(t)-\bu_2(t)\|_V+\|\varphi_1(t)-\varphi_2(t)\|_W).
\end{equation}
Here and below, $C$ stands for a positive constant depending on data
whose specific value may change from place to place. On the other
hand, from (\ref{3.9bis}) we know that
\begin{equation}\label{deuav}
\|\bu_1(t)-\bu_2(t)\|_V\le\int_0^t\|\bv_1(s)-\bv_2(s)\|_V\,ds.
\end{equation}
Also, from (\ref{eq_phi_eta}) and using (\ref{hip3}) and
(\ref{hip4}) we deduce
$$
\|\varphi_1(t)-\varphi_2(t)\|_W\le C\|\bu_1(t)-\bu_2(t)\|_V,
$$
which, combined with (\ref{deuav}), gives
\begin{equation}\label{dephiau}
\|\varphi_1(t)-\varphi_2(t)\|_W\le C\int_0^t\|\bv_1(s)-\bu_2(s)\|_V\,ds.
\end{equation}
Thus, by using consecutively (\ref{dephiau}) and (\ref{deuav}) in (\ref{eq_des0}), we obtain
$$
\|\Lambda\beeta_1(t)-\Lambda\beeta_2(t)\|_{V'}\le C\int_0^t\|\bv_1(s)-\bv_2(s)\|_V\,ds,
$$
which implies
\begin{equation}\label{eq_des1}
\|\Lambda\beeta_1(t)-\Lambda\beeta_2(t)\|_{V'}^2\le C\int_0^t\|\bv_1(s)-\bv_2(s)\|_V^2\,ds.
\end{equation}
Taking $\beeta=\beeta_i$ for $i=1,2$ successively in
(\ref{eq_u_eta}) with $\bw=\bv_1(t)-\bv_2(t)$ and subtracting the
resulting expressions, we find
\begin{align*}
&\hspace*{-1cm}<\rho\dot{\bv}_1(t)-\rho\dot{\bv}_2(t),\bv_1(t)-\bv_2(t)>_{V'\times V}\\
&\qquad\qquad \quad +(\A(\bvar(\bv_1(t))-\bvar(\bv_2(t))),\bvar(\bv_1(t))-\bvar(\bv_2(t)))_Q\\%
&\qquad\quad=<\beeta_2(t)-\beeta_1(t),\bv_1(t)-\bv_2(t)>_{V'\times
V}.
\end{align*}
By integrating in time, using the ellipticity of $\A$, the fact that
$\bv_1(0)=\bv_2(0)=\bv_0$ and Korn's inequality, we find
\begin{align*}
%\frac{1}{2}<\dot{\bv}_1-\dot{\bv}_2,\bv_1-\bv_2>_{V'\times V}+(\A(\var(\bv_1)-\var(\bv_2)),\var(\bv_1)-\var(\bv_2))_Q%
%=<\beeta_2-\beeta_1,\bv_1-\bv_2>_{V'\times V}.
&m_\A\int_0^t\|\bv_1(s)-\bv_2(s)\|_V^2ds\le\int_0^t<\beeta_2(s)-\beeta_1(s),\bv_1(s)-\bv_2(s)>_{V'\times V}ds\\%
&\qquad\le\frac{1}{m_\A}\int_0^t\|\beeta_2(s)-\beeta_1(s)\|_{V'}^2ds%
+\frac{m_\A}{4}\int_0^t\|\bv_1(s)-\bv_2(s)\|_V^2ds,
\end{align*}
where we used several times Young's inequality
\begin{equation}\label{Young}
ab\leq \epsilon a^2+\frac{1}{4\epsilon}b^2,\, a,b,\epsilon\in
\mathbb{R},\,\epsilon>0.
\end{equation}

 Plugging this into (\ref{eq_des1}) we find that
\begin{equation*}
\|\Lambda\beeta_1(t)-\Lambda\beeta_2(t)\|_{V'}^2\le
C\int_0^t\|\beeta_1(s)-\beeta_2(s)\|_{V'}^2\,ds,
\end{equation*}
and using a standard argument (see, for example, Lemma 4.7 in
\cite{SHSbook}), from the previous inequality and using the Banach's
fixed point theorem we conclude that there exists a unique
$\beeta^*\in L^2(0,T;V')$ such that $\Lambda\beeta*=\beeta*$.
\end{proof}

We can now give the proof of Theorem \ref{teo_PVM}.

\begin{proof}{[of Theorem \ref{teo_PVM}]}
By using Proposition \ref{teo_ptofijo_eta} there exists a unique $\beeta^*\in L^2(0,T;V')$ such that
$\Lambda\beeta^*=\beeta^*$.  We define $\bu_M=\bu_{M\eta^*}$, $\bv_M=\bv_{M\eta^*}$ and $\varphi_M=\varphi_{M\eta^*}$.
By taking $\beeta=\beeta^*$ in (\ref{eq_phi_eta}) we obtain (\ref{eq_phiM}). Also, from (\ref{def_lambda_eta}) we get
$$
<\rho\beeta^*(t),\bw>_{V'\times V}=(\B\bvar(\bu_{M}(t)),\bvar(\bw))_Q%
+(\E^*\nabla\varphi_{M}(t),\bvar(\bw))_Q+j(\bu_{M}(t),\bw),
$$
for all $\bw\in V$ and $t\in[0,T]$. Therefore, by taking
$\beeta=\beeta^*$ in (\ref{eq_u_eta}) and using the previous
equality, we obtain the variational equation (\ref{eq_vM}). Finally,
(\ref{def_uM}) follows from (\ref{3.9bis}), and the regularities are
a consequence of the regularities given by Propositions
\ref{teo_PVM_u_eta} and \ref{teo_PVM_phi_eta}.
\end{proof}
Further, we define the operator $\Theta:C([0,T];Q)\to C([0,T];Q)$ by
\begin{equation*}
\Theta\bF=\Ge(\bsigma_M,\bvar(\bu_M)),\ {\rm where}\
\bM=\int_0^t\bF(s)\,ds\quad \forall\bF\in C([0,T];Q),
\end{equation*}
and $\bu_M$ is the displacement field solution to Problem \ref{PVM}
while $\bsigma_M$ is the stress field:
\begin{equation}\label{def_sigmaM}
\bsigma_M=\A\bvar(\bv_M)+\B\bvar(\bu_M)+\bM+\E^*\nabla\varphi_M,
\end{equation}
with $\varphi_M$ being the electric potential solution to Problem
\ref{PVM}. Note that since $\bM\in C([0,T];Q)$, it is
straightforward that $\bsigma_M\in C([0,T];Q)$. We obtain the
following result.

\begin{proposition}\label{teo_ptofijo_M}
The operator $\Theta$ has a unique fixed point $\bF^*\in
C([0,T];Q)$.
\end{proposition}
\begin{proof}{}
The continuity of $\Theta\bF$ is a straightforward consequence of
the continuity of $\bsigma_M$ and $\bu_M$ and (\ref{hip6}).
Moreover, let $\bF_1,\bF_2\in C([0,T];Q)$ and let $\bM_1,\bM_2\in
C([0,T];Q)$ be their corresponding integrals in time. For the sake
of simplicity, we use the notation $\bu_i=\bu_{M_i}$,
$\bv_i=\bv_{M_i}$, $\bsigma_i=\bsigma_{M_i}$ and
$\varphi_i=\varphi_{M_i}$ for $i=1,2$. Let $t\in [0,T]$. From
(\ref{hip6}) we find that
\begin{equation}\label{eq_des3}
\|\Theta\bF_1(t)-\Theta\bF_2(t)\|_Q\le C(\|\bsigma_1(t)-\bsigma_2(t)\|_Q+\|\bu_1(t)-\bu_2(t)\|_V).
\end{equation}
By using (\ref{def_sigmaM}) in (\ref{eq_vM}) successively for
$\bM=\bM_i$, $i=1,2$, taking  in both cases $\bw=\bv_1(t)-\bv_2(t)$
and subtracting the resulting equations, we obtain
\begin{align*}
&<\rho(\dot{\bv}_1(t)-\dot{\bv}_2(t)),\bv_1(t)-\bv_2(t)>_{V'\times V}+(\bsigma_1(t)-\bsigma_2(t),\bvar(\bv_1(t)-\bv_2(t)))_Q\\%
&\qquad
+j(\bu_1(t),\bv_1(t)-\bv_2(t))-j(\bu_2(t),\bv_1(t)-\bv_2(t))=0.
\end{align*}
Integrating in time and using (\ref{2.16bis}) and
$\bv_1(0)=\bv_2(0)$ we deduce that
\begin{align*}
&\frac{1}{2}\|\bv_1(t)-\bv_2(t)\|_V^2\\%
&\qquad\le C\int_0^t(\|\bsigma_1(s)-\bsigma_2(s)\|_Q+\|\bu_1(s)-\bu_2(s)\|_V)\|\bv_1(s)-\bv_2(s)\|_V\,ds.
\end{align*}
Also, from (\ref{def_sigmaM}) it follows that
\begin{align}
&\|\bsigma_1(t)-\bsigma_2(t)\|_Q\le C\left(\|\bv_1(t)-\bv_2(t)\|_V+\|\bu_1(t)-\bu_2(t)\|_V\right.\nonumber\\
&\left.\qquad\qquad
+\|\varphi_1(t)-\varphi_2(t)\|_W+\int_0^t\|\bF_1(s)-\bF_2(s)\|_Q\,ds\right).\label{eq_des2}
\end{align}
Combining the last two inequalities, we get
\begin{align*}
&\|\bv_1(t)-\bv_2(t)\|_V^2%
\le C\int_0^t\left(\|\bv_1(s)-\bv_2(s)\|_V+\|\bu_1(s)-\bu_2(s)\|_V\right.\\%
&\left.\qquad+\|\varphi_1(s)-\varphi_2(s)\|_W+\int_0^s\|\bF_1(r)-\bF_2(r)\|_Q\,dr\right)\|\bv_1(s)-\bv_2(s)\|_V\,ds.
\end{align*}
By using (\ref{deuav}) and (\ref{dephiau}) and after some tedious calculations, we obtain
\begin{equation*}
\|\bv_1(t)-\bv_2(t)\|_V^2%
\le C\left(\int_0^t(\|\bv_1(s)-\bv_2(s)\|_V^2\,ds%
+\int_0^t\|\bF_1(s)-\bF_2(s)\|_Q^2\,ds\right).
\end{equation*}
By using the Gronwall's Lemma, we find that
\begin{equation*}
\|\bv_1(t)-\bv_2(t)\|_V^2\le C\int_0^t\|\bF_1(s)-\bF_2(s)\|_Q^2\,ds.
\end{equation*}
This last inequality combined with (\ref{deuav}), (\ref{dephiau})
and (\ref{eq_des2}) allows us to have, from (\ref{eq_des3}), the
following estimate:
\begin{equation*}
\|\Theta\bF_1(t)-\Theta\bF_2(t)\|_Q^2\le C\int_0^t\|\bF_1(s)-\bF_2(s)\|_Q^2\,ds.
\end{equation*}
Following a standard argument (see again Lemma 4.7 in
\cite{SHSbook}), from the previous inequality and using the Banach's
fixed point theorem, we conclude that there exists a unique
$\bF^*\in C([0,T];Q)$ such that $\Theta\bF^*=\bF^*$.
\end{proof}
We can now give the proof of Theorem \ref{teo_PV}.

\begin{proof}{[of Theorem \ref{teo_PV}]}
By using Proposition \ref{teo_ptofijo_M} there exists a unique
$\bF^*\in C([0,T];Q)$ such that $\Theta\bF^*=\bF^*$.  We define
$$
\bM^*(t)=\int_0^t\bF^*(s)\,ds.
$$
We also define $\bu=\bu_{M^*}$, $\bv=\bv_{M^*}$,
$\bsigma=\bsigma_{M^*}$ and $\varphi=\varphi_{M^*}$. By taking
$\bM=\bM^*$ in (\ref{eq_vM}) we obtain (\ref{var2}), because
$$
\bM(t)=\bM^*(t)=\int_0^t\bF^*(s)\,ds=\int_0^t\Theta\bF^*(s)\,ds=\int_0^t\Ge(\bsigma(s),\bvar(\bu(s)))\,ds.
$$
Finally, from (\ref{eq_phiM}) it follows (\ref{var3}), and
 (\ref{def_sigmaM}) leads to (\ref{var1}).
\end{proof}

\section{Fully discrete approximations and an a priori error analysis}\label{section4}\setcounter{equation}{0}

In this section, we introduce a finite element algorithm for
approximating solutions to variational problem \ref{PV}. Its
discretization is done in two steps. First, we consider the finite
element spaces $V^h\subset V$, $Q^h\subset Q$ and $W^h\subset W$
given by
\begin{eqnarray}
&&\hspace*{-0.7cm}\label{defvh} V^h=\{\bv^h\in
[C(\overline{\Omega})]^d \; ; \; \bv^h_{|_{T}}\in [P_1(T)]^d, \,
T\in {\T}^h,\quad
\bv^h=\bzero\, \hbox{    on   } \,\Gamma_D\},\\[2pt]
&&\label{defqh}\hspace*{-0.7cm} Q^h=\{\btau^h\in Q \; ; \;
\btau^h_{|_{T}}\in [P_0(T)]^{d\times d}, \,   T\in {\T}^h\},\\[2pt]
&&\hspace*{-0.7cm}\label{defwh} W^h=\{\psi^h\in C(\overline{\Omega})
\; ; \; \psi^h_{|_{T}}\in P_1(T), \,   T\in {\T}^h,\quad \psi^h=0\,
\hbox{    on   } \,\Gamma_A\},
\end{eqnarray}
where $\Omega$ is assumed to be a polyhedral domain, ${\T}^h$
denotes a triangulation of $\overline{\Omega}$ compa\-ti\-ble with
the partition of the boundary $\Gamma=\partial \Omega$ into
$\Gamma_D$, $\Gamma_N$ and $\Gamma_C$ on one hand, and into
$\Gamma_A$ and $\Gamma_B$ on the other hand, and $P_q(T)$, $q=0,1$,
represents the space of polynomials of global degree less or equal
to $q$ in $T$. Here, $h>0$ denotes the spatial discretization
parameter.

Secondly, the time derivatives are discretized by using a uniform
partition of the time interval $[0,T]$, denoted by
$0=t_0<t_1<\ldots<t_N=T$, and let $k$ be the time step size,
$k=T/N$. Moreover, for a continuous function $f(t)$ we denote
$f_n=f(t_n)$ and, for the sequence $\{z_n\}_{n=0}^N$, we denote by
$\delta z_n=(z_n-z_{n-1})/k $ its corresponding divided differences.

Using a hybrid combination of the forward and backward Euler
schemes, the fully discrete approximation of Problem \ref{PV} is the
following.

\begin{problem}\label{PVhk}
Find a discrete velocity field
$\bv^{hk}=\{\bv_n^{hk}\}_{n=0}^N\subset V^h$, a discrete stress
field $\bsigma^{hk}=\{\bsigma_n^{hk}\}_{n=0}^N\subset Q^h$ and a
discrete electric potential field
$\varphi^{hk}=\{\varphi_n^{hk}\}_{n=0}^N\subset W^h$ such that
$\bv^{hk}_0=\bv_0^h$ and for $n=1,\ldots, N$ and for all $\bw^h\in
V^h$ and $\psi^h\in W^h$,
\begin{eqnarray}
&&\hspace*{-0.6cm}
\bsigma_n^{hk}=\A\bvarepsilon(\bv_n^{hk})+\B\bvarepsilon(\bu_n^{hk})+k\sum_{j=0}^{n-1}
\Ge (\bsigma_j^{hk},\bvarepsilon(\bu_j^{hk}))- \E^*{\bf
E}(\varphi_{n-1}^{hk}),\label{disvar1}\\
&&\hspace*{-0.6cm} \displaystyle \nonumber ( \rho
\delta{\bv_n^{hk}},\bw^h)_H
+\left(\A\bvarepsilon(\bv_n^{hk})+\B\bvarepsilon(\bu_n^{hk})+k\sum_{j=0}^{n-1}\Ge (\bsigma_j^{hk},\bvarepsilon(\bu_j^{hk})),\bvarepsilon(\bw^h)\right)_Q\\
&&\qquad \qquad =\langle\fb_n,\bw^h\rangle_{V'\times V}-
\left(\E^*{\bf E}(\varphi_{n-1}^{hk}),\bvarepsilon(\bw^h)\right)_Q -
j(\bu_{n-1}^{hk},\bw^h) , \label{disvar2}\\
&& \hspace*{-0.6cm} \label{disvar3} (\Beta \nabla
\varphi_n^{hk},\nabla \psi^h)_H-(\E\bvarepsilon(\bu_n^{hk}),\nabla
\psi^h)_H= (q_n,\psi^h)_W,
\end{eqnarray}
where the discrete displacement field
$\bu^{hk}=\{\bu_n^{hk}\}_{n=0}^N\subset V^h$ is given by
\begin{equation}\label{disvar4}
\bu_n^{hk}=k\sum_{j=1}^n \bv_j^{hk} + \bu_0^h,
\end{equation}
and the artificial discrete initial condition $\varphi_0^{hk}$ is
the solution to the following problem:
\begin{equation}\label{disvar5}
(\Beta \nabla \varphi_0^{hk},\nabla
\psi^h)_H-(\E\bvarepsilon(\bu_0^{h}),\nabla \psi^h)_H=
(q_0,\psi^h)_W\quad \forall \psi^h\in W^h.
\end{equation}
\end{problem}
Here, we note that the discrete initial conditions, denoted by
$\bu_0^h$ and $\bv_0^h$ are given by
\begin{equation}\label{dis_inicond}
\bu_0^h=\mathcal{P}^{h} \bu_0,\quad \bv_0^h=\mathcal{P}^{h} \bv_0,
\end{equation}
where $\mathcal{P}^{h}$ is the $[L^2(\Omega)]^d$-projection operator
over the finite element space $V^h$.

Using assumptions (\ref{hip1})--(\ref{hip8})  and the classical
Lax-Milgram lemma, it is easy to prove that Problem \ref{PVhk} has a
unique discrete solution
$(\bv^{hk},\varphi^{hk},\bsigma^{hk})\subset V^h\times W^h\times
Q^h$.

Our aim in this section is to derive some a priori error estimates
for the numerical errors $\bu_n-\bu_n^{hk}$, $\bv_n-\bv_n^{hk}$ and
$\varphi_n-\varphi_n^{hk}$. Therefore, we assume that the solution
to Problem \ref{PV} has the following regularity:
\begin{equation}\label{regu}
\begin{array}{l}
\bu\in C^1([0,T];V)\cap C^2([0,T];H),\quad \bsigma \in C([0,T];Q),\\
\varphi\in C([0,T];W).
\end{array}
\end{equation}

Thus, we have the following result.

\begin{theorem}\label{theorem_apriori}
Let assumptions (\ref{hip1})--(\ref{hip8}) and the additional
regularity (\ref{regu}) hold. If we denote by
$(\bv,\varphi,\bsigma)$ and $(\bv^{hk},\varphi^{hk},\bsigma^{hk})$
the respective solutions to problems \ref{PV} and \ref{PVhk}, then
there exists a positive constant $C>0$, independent of the
discretization parameters $h$ and $k$, such that, for all
$\bw^h=\{\bw_j^h\}_{j=0}^N\subset  V^h$ and
$\psi^h=\{\psi_j^h\}_{j=0}^N\subset W^h$,
\begin{eqnarray}
&&\nonumber\hspace*{-1cm} \displaystyle \max_{0\leq n \leq
N}\|\bv_n-\bv_n^{hk}\|_H^2+ \max_{0\leq n \leq
N}\|\bu_n-\bu_n^{hk}\|_V^2+\max_{0\leq n \leq
N}\|\varphi_n-\varphi_n^{hk}\|_W^2\\
&& \qquad
+Ck\sum_{j=0}^N\|\bv_j-\bv_j^{hk}\|_V^2+Ck\sum_{j=0}^N\|\bsigma_j-\bsigma_j^{hk}\|_Q^2\nonumber
\end{eqnarray}
\begin{eqnarray}
&& \leq Ck\sum_{j=1}^N\Big(\|\dot{\bv}_j-\delta
\bv_j\|_H^2+\|\dot{\bu}_j-\delta \bu_j\|_V^2
 +k^2+\|\bv_j-\bw^h_j\|_V^2+\|\varphi_j-\psi_j^h\|_W^2\nonumber\\
&& \nonumber \qquad +I_j^2\Big)+C\max_{0\leq n\leq
N}\|\bv_n-\bw_n^h\|_H^2+C\Big(\|\bu_0-\bu_0^h\|_V^2+\|\bv_0-\bv_0^h\|_H^2\\
&& \qquad +C\|\varphi_0-\psi_0^h\|_W^2\Big)
+\frac{C}{k}\sum_{j=1}^{N-1}\|\bv_j-\bw_j^h-(\bv_{j+1}-\bw_{j+1}^h)\|^2_H,
\label{est_final}
\end{eqnarray}
where the integration error $I_n$ is defined as
\begin{equation}\label{defIn}
I_n=\left\|\int_0^{t_n} \Ge(\bsigma(s),\bvarepsilon(\bu(s)))\,
ds-k\sum_{j=0}^{n-1}\Ge (\bsigma_j,\bvarepsilon(\bu_j))\right\|_Q.
\end{equation}
\end{theorem}
\begin{proof}

First, we obtain some estimates on the stress field. Subtracting
equations (\ref{var1}), at time $t=t_n$, and (\ref{disvar1}), taking
into account assumptions (\ref{hip1})--(\ref{hip8}) we easily find
that
\begin{equation}\label{est_stress}
\begin{array}{l}
\hspace*{-0cm}\displaystyle \|\bsigma_n-\bsigma_n^{hk}\|_Q\leq
C\Big(\|\bv_n-\bv_n^{hk}\|_V+ \|\bu_n-\bu_n^{hk}\|_V+
I_n+k\\
\qquad\qquad\displaystyle
+k\sum_{j=0}^{n-1}\Big[\|\bu_{j}-\bu_{j}^{hk}\|_V+\|\bsigma_j-\bsigma_j^{hk}\|_Q+\|\varphi_j-\varphi_{j}^{hk}\|_W\Big]\Big),
\end{array}
\end{equation}
where the integration error $I_n$ is defined in (\ref{defIn}).

Secondly, we obtain the estimates on the electric potential field.
We subtract variational equation (\ref{var3}), at time $t=t_n$ and
for $\psi=\psi^h\in W^h$, and discrete variational equation
(\ref{disvar3}) to get, for all $\psi^h\in W^h$,
$$(\Beta \nabla
(\varphi_n-\varphi_n^{hk}),\nabla
\psi^h)_H-(\E\bvarepsilon(\bu_n-\bu_n^{hk}),\nabla \psi^h)_H=0\quad
\forall \psi^h\in W^h.$$ Therefore, it follows that, for all $
\psi^h\in W^h$,
$$\begin{array}{l}
(\Beta \nabla (\varphi_n-\varphi_n^{hk}),\nabla
(\varphi_n-\varphi_n^{hk}))_H-(\E\bvarepsilon(\bu_n-\bu_n^{hk}),\nabla
(\varphi_n-\varphi_n^{hk}))_H\\
\qquad = (\Beta \nabla (\varphi_n-\varphi_n^{hk}),\nabla
(\varphi_n-\psi^h))_H-(\E\bvarepsilon(\bu_n-\bu_n^{hk}),\nabla
(\varphi_n-\psi^h))_H.
\end{array}$$

Using again assumptions (\ref{hip1})--(\ref{hip8}) and several times
Young's inequality (\ref{Young}), we find that
\begin{equation}\label{est_elec_pot}
\|\varphi_n-\varphi_n^{hk}\|_W^2\leq
C\Big(\|\bu_n-\bu_n^{hk}\|_V^2+\|\varphi_n-\psi^h\|_W^2\Big) \quad
\forall \psi^h\in W^h.
\end{equation}

Finally, we obtain the estimates on the velocity and displacement
fields. To do that, we subtract variational equation (\ref{var2}),
at time $t=t_n$ and for $\bw=\bw^h\in V^h$, and discrete variational
equation (\ref{disvar2}) to get, for all $\bw^h\in V^h$,
$$\begin{array}{l}
\hspace*{-0.5cm}\displaystyle ( \rho
(\dot{\bv}_n-\delta{\bv_n^{hk}}),\bw^h)_H
+\left(\A\bvarepsilon(\bv_n-\bv_n^{hk})+\B\bvarepsilon(\bu_n-\bu_n^{hk}),\bvarepsilon(\bw^h)\right)_Q\\
\quad \displaystyle +\left(\int_0^{t_n} \Ge
(\bsigma(s),\bvarepsilon(\bu(s)))\, ds- k\sum_{j=1}^{n-1}\Ge
(\bsigma_j,\bvarepsilon(\bu_j)),\bvarepsilon(\bw^h)\right)_Q\\
\quad \displaystyle +\left( k\sum_{j=0}^{n-1}[\Ge (\bsigma_j,\bvarepsilon(\bu_j))-\Ge (\bsigma_j^{hk},\bvarepsilon(\bu_j^{hk}))],\bvarepsilon(\bw^h)\right)_Q\\
\quad - \left(\E^*{\bf E}(\varphi_{n})-\E^*{\bf
E}(\varphi_{n-1}^{hk}),\bvarepsilon(\bw^h)\right)_Q +j(\bu_n,\bw^h)-
j(\bu_{n-1}^{hk},\bw^h)=0.
\end{array}$$
Therefore, we find that, for all  $\bw^h\in V^h$,
$$\begin{array}{l}
\displaystyle ( \rho
(\dot{\bv}_n-\delta{\bv_n^{hk}}),\bv_n-\bv_n^{hk})_H
+\left(\A\bvarepsilon(\bv_n-\bv_n^{hk})+\B\bvarepsilon(\bu_n-\bu_n^{hk}),\bvarepsilon(\bv_n-\bv_n^{hk})\right)_Q\\
\qquad \displaystyle +\left(\int_0^{t_n} \Ge
(\bsigma(s),\bvarepsilon(\bu(s)))\, ds- k\sum_{j=0}^{n-1}\Ge
(\bsigma_j,\bvarepsilon(\bu_j)),\bvarepsilon(\bv_n-\bv_n^{hk})\right)_Q\\
\qquad \displaystyle +\left( k\sum_{j=0}^{n-1}[\Ge (\bsigma_j,\bvarepsilon(\bu_j))-\Ge (\bsigma_j^{hk},\bvarepsilon(\bu_j^{hk}))],\bvarepsilon(\bv_n-\bv_n^{hk})\right)_Q\\
\qquad - \left(\E^*{\bf E}(\varphi_{n})-\E^*{\bf
E}(\varphi_{n-1}^{hk}),\bvarepsilon(\bv_n-\bv_n^{hk})\right)_Q\\
\qquad  +j(\bu_n,\bv_n-\bv_n^{hk})-
j(\bu_{n-1}^{hk},\bv_n-\bv_n^{hk})\\
\displaystyle\quad = ( \rho
(\dot{\bv}_n-\delta{\bv_n^{hk}}),\bv_n-\bw^{h})_H
+\left(\A\bvarepsilon(\bv_n-\bv_n^{hk})+\B\bvarepsilon(\bu_n-\bu_n^{hk}),\bvarepsilon(\bv_n-\bw^{h})\right)_Q\\
\qquad \displaystyle +\left(\int_0^{t_n} \Ge
(\bsigma(s),\bvarepsilon(\bu(s)))\, ds- k\sum_{j=0}^{n-1}\Ge
(\bsigma_j,\bvarepsilon(\bu_j)),\bvarepsilon(\bv_n-\bw^h)\right)_Q\\
\qquad \displaystyle +\left( k\sum_{j=0}^{n-1}[\Ge (\bsigma_j,\bvarepsilon(\bu_j))-\Ge (\bsigma_j^{hk},\bvarepsilon(\bu_j^{hk}))],\bvarepsilon(\bv_n-\bw^{h})\right)_Q\\
\qquad - \left(\E^*{\bf E}(\varphi_{n})-\E^*{\bf
E}(\varphi_{n-1}^{hk}),\bvarepsilon(\bv_n-\bw^{h})\right)_Q\\
\qquad  +j(\bu_n,\bv_n-\bw^{h})- j(\bu_{n-1}^{hk},\bv_n-\bw^{h}).
\end{array}$$

Keeping in mind assumptions (\ref{hip1}) and (\ref{hip2}) it follows
that
$$\begin{array}{l}
\displaystyle\left(\A\bvarepsilon(\bv_n-\bv_n^{hk}),\bvarepsilon(\bv_n-\bv_n^{hk})\right)_Q\geq
C\|\bv_n-\bv_n^{hk}\|_V^2,\\
\displaystyle\left(\B\bvarepsilon(\bu_n-\bu_n^{hk}),\bvarepsilon(\bv_n-\bv_n^{hk})\right)_Q\geq
\displaystyle\left(\B\bvarepsilon(\bu_n-\bu_n^{hk}),\bvarepsilon(\dot{\bu}_n-\delta\bu_n)\right)_Q\\
\displaystyle\qquad +
\frac{C}{2k}\left\{\|\bu_n-\bu_n^{hk}\|_V^2-\|\bu_{n-1}-\bu_{n-1}^{hk}\|_V^2\right\},
\end{array}$$
where $\delta \bu_n=(\bu_n-\bu_{n-1})/k$ and we used
(\ref{disvar4}). Moreover, since
$$\begin{array}{l}
\displaystyle\hspace*{-1cm}(
\rho(\dot{\bv}_n-\delta{\bv_n^{hk}}),\bv_n-\bv_n^{hk})_H\geq (
\rho(\dot{\bv}_n-\delta{\bv_n}),\bv_n-\bv_n^{hk})_H\\
\displaystyle\qquad +
\frac{C}{2k}\left\{\|\bv_n-\bv_n^{hk}\|_H^2-\|\bv_{n-1}-\bv_{n-1}^{hk}\|_H^2\right\},
\end{array}$$
where $\delta \bv_n=(\bv_n-\bv_{n-1})/k$, using again several times
Young's inequality (\ref{Young}) and assumptions
(\ref{hip1})--(\ref{hip8}), we have
\begin{eqnarray*}
&&\nonumber\hspace*{-1cm} \displaystyle  \frac{1}{2k}\left[
\|\bv_n-\bv_n^{hk}\|_H^2-\|\bv_{n-1}-\bv_{n-1}^{hk}\|_H^2\right] +
\frac{1}{2k}\left[
\|\bu_n-\bu_n^{hk}\|_V^2-\|\bu_{n-1}-\bu_{n-1}^{hk}\|_V^2\right]\\
&&\nonumber \qquad  +C\|\bv_n-\bv_n^{hk}\|_V^2\\
&&\nonumber   \leq C\Big(\|\dot{\bv}_n-\delta
\bv_n\|_H^2+\|\dot{\bu}_n-\delta \bu_n\|_V^2
 +k^2+\|\bv_{n}-\bv_{n}^{hk}\|_H^2+\|\bv_n-\bw^h\|_V^2\\
&&\nonumber  \qquad
 +\|\varphi_{n-1}-\varphi_{n-1}^{hk}\|_W^2+I_n^2+\|\bu_n-\bu_n^{hk}\|_V^2+\|\bu_{n-1}-\bu_{n-1}^{hk}\|_V^2\\
&& \qquad  +\rho(\delta \bv_n-\delta
 \bv_n^{hk},\bv_n-\bw^h)_H\Big) \quad \forall \bw^h\in
 V^h.
\end{eqnarray*}
Therefore, by induction we find that, for all
$\bw^h=\{\bw_j^h\}_{j=0}^n\subset
 V^h$,
\begin{eqnarray}
&&\nonumber\hspace*{-0.7cm} \displaystyle \|\bv_n-\bv_n^{hk}\|_H^2+
\|\bu_n-\bu_n^{hk}\|_V^2+Ck\sum_{j=1}^n\|\bv_j-\bv_j^{hk}\|_V^2\\
&&\nonumber \hspace*{-0.2cm} \leq
Ck\sum_{j=1}^n\Big(\|\dot{\bv}_j-\delta
\bv_j\|_H^2+\|\dot{\bu}_j-\delta \bu_j\|_V^2
 +k^2+\|\bv_{j}-\bv_{j}^{hk}\|_H^2\\
&&\nonumber  \quad +\|\bv_j-\bw^h_j\|_V^2
 +\|\varphi_{j-1}-\varphi_{j-1}^{hk}\|_W^2+I_j^2+\|\bu_j-\bu_j^{hk}\|_V^2\\
&& \quad  +\rho(\delta \bv_j-\delta
 \bv_j^{hk},\bv_j-\bw_j^h)_H\Big)+C\left( \|\bv_0-\bv_0^h\|_H^2+\|\bu_0-\bu_0^h\|_V^2\right)
 .\label{est_vel}
\end{eqnarray}

Now, combining (\ref{est_stress}), (\ref{est_elec_pot}) and
(\ref{est_vel}) it follows that, for all
$\bw^h=\{\bw_j^h\}_{j=0}^n\subset  V^h$ and
$\psi^h=\{\psi_j^h\}_{j=0}^n\subset W^h$,
\begin{eqnarray*}
&&\nonumber\hspace*{-0.7cm} \displaystyle \|\bv_n-\bv_n^{hk}\|_H^2+
\|\bu_n-\bu_n^{hk}\|_V^2+\|\varphi_n-\varphi_n^{hk}\|_W^2\\
&& \qquad +Ck\sum_{j=1}^n\|\bv_j-\bv_j^{hk}\|_V^2+Ck\sum_{j=1}^n\|\bsigma_j-\bsigma_j^{hk}\|_Q^2\\
&& \leq Ck\sum_{j=1}^n\Big(\|\dot{\bv}_j-\delta
\bv_j\|_H^2+\|\dot{\bu}_j-\delta \bu_j\|_V^2
 +k^2+\|\bv_{j}-\bv_{j}^{hk}\|_H^2\\
&&\nonumber  \qquad +\|\bv_j-\bw^h_j\|_V^2
 +\|\varphi_{j-1}-\varphi_{j-1}^{hk}\|_W^2+I_j^2+\|\bu_j-\bu_j^{hk}\|_V^2+\|\varphi_j-\psi_j^h\|_W^2\\
&& \qquad  +\rho(\delta \bv_j-\delta
 \bv_j^{hk},\bv_j-\bw_j^h)_H+k\sum_{l=1}^j\|\bsigma_l-\bsigma_l^{hk}\|_Q^2\Big)\nonumber\\
 && \qquad +C\left( \|\bv_0-\bv_0^h\|_H^2+\|\bu_0-\bu_0^h\|_V^2\right).
\end{eqnarray*}

Now, taking into account that
$$\begin{array}{l}
\displaystyle k\sum_{j=1}^n \rho(\delta \bv_j-\delta
 \bv_j^{hk},\bv_j-\bw^h_j)_H=\sum_{j=1}^n
\rho(\bv_{j}-\bv_{j}^{hk}-(\bv_{j-1}-\bv_{j-1}^{hk}),\bv_j-\bw_j^h)_H\\
 \displaystyle
\qquad
=\rho(\bv_n-\bv_n^{hk},\bv_n-\bw_n^h)_H+\rho(\bv_0^h-\bv_0,\bv_1-\bw_1^h)_H\\
\displaystyle \qquad\qquad + \sum_{j=1}^{n-1}
\rho(\bv_j-\bv_j^{hk},\bv_j-\bw_j^h-(\bv_{j+1}-\bw_{j+1}^h))_H,
\end{array}$$
using once again Young's inequality (\ref{Young}) we have, for all
$\bw^h=\{\bw_j^h\}_{j=0}^n\subset V^h$ and
$\psi^h=\{\psi_j^h\}_{j=0}^n\subset W^h$,
\begin{eqnarray*}
&&\nonumber\hspace*{-1cm} \displaystyle \|\bv_n-\bv_n^{hk}\|_H^2+
\|\bu_n-\bu_n^{hk}\|_V^2+\|\varphi_n-\varphi_n^{hk}\|_W^2\\
&& \qquad +Ck\sum_{j=1}^n\|\bv_j-\bv_j^{hk}\|_V^2+Ck\sum_{j=1}^n\|\bsigma_j-\bsigma_j^{hk}\|_Q^2\\
&& \leq Ck\sum_{j=1}^n\Big(\|\dot{\bv}_j-\delta
\bv_j\|_H^2+\|\dot{\bu}_j-\delta \bu_j\|_V^2
 +k^2+\|\bv_{j}-\bv_{j}^{hk}\|_H^2+I_j^2\\
&&\nonumber  \qquad +\|\bv_j-\bw^h_j\|_V^2
 +\|\varphi_{j-1}-\varphi_{j-1}^{hk}\|_W^2+\|\bu_j-\bu_j^{hk}\|_V^2+\|\varphi_j-\psi_j^h\|_W^2\\
&& \qquad +k\sum_{l=1}^j\|\bsigma_l-\bsigma_l^{hk}\|_Q^2\Big)
 +\frac{C}{k}\sum_{j=1}^{n-1}\|\bv_j-\bw_j^h-(\bv_{j+1}-\bw_{j+1}^h)\|^2_H\\
&& \qquad
+C\left(\|\bv_0-\bv_0^h\|_H^2+\|\bu_0-\bu_0^h\|_V^2+\|\bv_1-\bw_1^h\|_H^2+\|\bv_n-\bw^h_n\|_H^2\right)
.
\end{eqnarray*}

From the regularity (\ref{regu}) we conclude that $\varphi(0)$ is
the solution to the fo\-llo\-wing problem:
$$(\Beta \nabla \varphi(0),\nabla
\psi_0)_H-(\E\bvarepsilon(\bu(0)),\nabla \psi_0)_H=
(q(0),\psi_0)_W\quad \forall \psi_0\in W,$$ and so, proceeding as in
the proof of estimates (\ref{est_elec_pot}), we easily find that
\begin{equation*}
\|\varphi_0-\varphi_0^{hk}\|_W^2\leq
C\Big(\|\bu_0-\bu_0^{h}\|_W^2+\|\varphi_0-\psi_0^h\|_W^2\Big) \quad
\forall \psi_0^h\in W^h.
\end{equation*}

Finally, using a discrete version of Gronwall's inequality (see, for
instance, \cite{CFHS05}) we derive the a priori error estimates
(\ref{est_final}).
\end{proof}

We note that from estimates (\ref{est_final}) we can derive the
convergence order under suitable additional regularity conditions.
For instance, if we assume that the continuous solution has the
additional regularity:
 \begin{equation}\label{adi_regu}\begin{array}{l}
 \bu\in H^2(0,T;V)\cap H^3(0,T;H)\cap
 C^1([0,T];[H^2(\Omega)]^d),\\
 \varphi\in C([0,T];H^2(\Omega)),\end{array}
 \end{equation}
then we have the following result.
\begin{corollary}
Let the assumptions of Theorem \ref{theorem_apriori} still hold.
Under the additional regularity conditions (\ref{adi_regu}), it
follows the linear convergence of the solution obtained by Problem
\ref{PVhk}; that is, there exists a positive constant $C$,
independent of the discretization parameters $h$ and $k$, such that
\begin{eqnarray*}
&&\nonumber\hspace*{-1cm} \displaystyle \max_{0\leq n \leq
N}\|\bv_n-\bv_n^{hk}\|_H+ \max_{0\leq n \leq
N}\|\bu_n-\bu_n^{hk}\|_V+\max_{0\leq n \leq
N}\|\varphi_n-\varphi_n^{hk}\|_W\leq C(h+k).
\end{eqnarray*}
\end{corollary}

Notice that this linear convergence is based on some well-known
results concerning the appro\-xi\-ma\-tion by the finite element
method (see, for instance, \cite{Ciarlet}), the discretization of
the time derivatives and the following result (see
\cite{BFH05,CFKSV} for details),
$$
\displaystyle
\frac{1}{k}\sum_{j=1}^{N-1}\|\bv_j-\mathcal{P}^h\bv_j-(\bv_{j+1}-\mathcal{P}^h\bv_{j+1})\|^2_H
\leq Ch^2\|\bu\|_{H^2(0,T;V)}^2.$$ Moreover, from the approximation
properties of operator $\mathcal{P}^h$, taking into account
regularities (\ref{adi_regu}) we can easily find that
$$\|\bu_0-\mathcal{P}^h\bu_0\|_V^2+\|\bv_0-\mathcal{P}^h\bv_0\|_H^2 \leq
Ch^2.$$

\section{Numerical results} \label{section5}

In order to verify the behaviour of the numerical method analyzed in
the previous section, some numerical experiments have been performed
in two-dimensional problems.
In all the examples presented, the elastic tensor %${\cal E}\bvarepsilon(\bu)$
was chosen as the $2D$ plane-stress elasticity tensor,
\begin{equation}
(\B \btau)_{\alpha \beta}=\frac{E r}{1-r^2}
(\tau_{11}+\tau_{22})\delta_{\alpha\beta}+\frac{E}{1+r}\tau_{\alpha\beta}
\quad \forall \btau \in \mathbb{S}^2, \label{2dps}
\end{equation}
where $\alpha,\beta=1,2$, $E$ and $r$ are the Young's modulus and
the Poisson's ratio, respectively, and $\delta_{\alpha\beta}$
denotes the Kronecker symbol. The viscous tensor $\A$ has a similar
form but multiplied by a damping coefficient $10^{-2},$ i.e.
$\mathcal{A}=10^{-2}\,\mathcal{B}$.

The viscoplastic function is a version of the Maxwell function given
by
\begin{equation}
\Ge(\bsigma, \bvarepsilon(\bu))= -\frac{1}{100}\Phi(\bsigma),
\end{equation}
being $\Phi$ %i:\Sd\rightarrow \Sd$
a truncation operator defined as
\begin{equation*}
\forall \btau =(\tau_{\alpha\beta})_{\alpha,\beta=1}^2, \quad
(\Phi(\btau))_{\alpha\beta}= \left\{\begin{array}{l}
L \quad \hbox{if} \quad \tau_{\alpha\beta}>L,\\
\tau_{\alpha\beta} \quad \hbox{if} \quad \tau_{\alpha\beta}\in [-L,L],\\
-L \quad \hbox{if} \quad \tau_{\alpha\beta}<-L,
\end{array}\right.
\end{equation*}
where value $L=1000$ was taken.

Moreover, as piezoelectric and permittivity tensors,  the following
matricial forms were considered:
\begin{equation}
e_{ijk}\equiv e_{pq}= \left(\begin{array}{lcr}
0 & 0 & e_{13} \\
e_{21} & e_{22} & 0
\end{array}\right), \qquad \qquad \beta_{ij}=\left(\begin{array}{lr}
\beta_{11} & 0  \\
0 & \beta_{22}
\end{array}\right),
\end{equation}
where we have used the notations $e_{ijk}$ and $e_{pq}$ in such a
way that $p=i$ and $q=1$ if $(i,j)=(1,1)$, $q=2$ if $(i,j)=(2,2)$
and $q= 3$ if $(i,j)=(1,2)$ or $(i,j)=(2,1)$.

Finally, in all the examples the normal compliance function is
defined as $$p(r)=c_p\,\max\{0,r\},$$ where $c_p>0$ is a
deformability coefficient.

\subsection{Numerical scheme}

As a first step, the artificial discrete initial condition for the
electric potential field is obtained by solving equation
(\ref{disvar5}). This leads to a linear symmetric system solved by
using classical Cholesky's method.
%\ref{} and ref{}.

Secondly, being the solution $\bu_{n-1}^{hk}, \bv_{n-1}^{hk}$ and $\varphi_{n-1}^{hk}$ known at time $t_{n-1},$
the velocity field is obtained by solving the discrete equation
$$\begin{array}{l}
\displaystyle (\rho\bv_n^{hk},
\bw^h)_H+k\left(\A\bvarepsilon(\bv_n^{hk})+k\B\bvarepsilon(\bv_n^{hk}),
\bvarepsilon(\bw^h)  \right)_Q
=(\rho\bv_{n-1}^{hk}, \bw^h)_H+ k\left< \fb_n,\bw^h\right>_{V'\times V} \\
\displaystyle -kj(\bu_{n-1}^{hk},\bw^h)
+k\left(\E^*\bE(\varphi_{n-1}^{hk}),\bvarepsilon(\bw^h)\right)_Q -
\left(k\B\bvarepsilon(\bu_{n-1}^{hk}),\bvarepsilon(\bw^h) \right)_Q \\
\displaystyle -k^2\left(\sum_{j=1}^{n-1} \Ge
\left(\bsigma_j^{hk},\bvarepsilon(\bu_j^{hk})\right),\bvarepsilon(\bw^h)\right)_Q,
\end{array}$$
where the decomposition
\begin{equation}\label{actdesp}
\bu_n^{hk}=\bu_{n-1}^{hk}+k\bv_n^{hk},
\end{equation}
has been used. Later, displacement field $\bu_{n}^{hk}$ is updated
through expression (\ref{actdesp}), and the electric potential field
is obtained solving equation (\ref{disvar3}). We note that both
numerical problems lead to linear symmetric systems and therefore,
 Cholesky's method is applied again for their solution.

The numerical scheme was implemented on a Intel Core $i5-3337U$ @
$1.80$GHz using FreeFEM$++$  (see \cite{FreeFEM} for details) and a
typical run ($100$ step times and $10000$ nodes) took about $3$ min
of CPU time.

\subsection{A first example: numerical convergence}

As a first example, a sequence of numerical solutions, based on
uniform partitions of both the time interval $[0,1]$ and the domain
$\Omega=[0,1] \times [0,1],$ have been performed in order to check
the behaviour of the numerical scheme.

The physical setting of the example is depicted in Figure
\ref{f-test1} (left-hand side). The body is in initial contact with
a deformable foundation on its lower part while $\Gamma_D=\Gamma_A$
is the right boundary $ \{1\}\times [0,1]$ (and so both the
displacement and electric potential fields vanish there). A surface
force acts on the upper surface $[0,1]\times \{1\}$ and no electric
charges are applied nor in the body or the surface.

\begin{figure}[h]
\begin{center}
\includegraphics[scale=0.58]{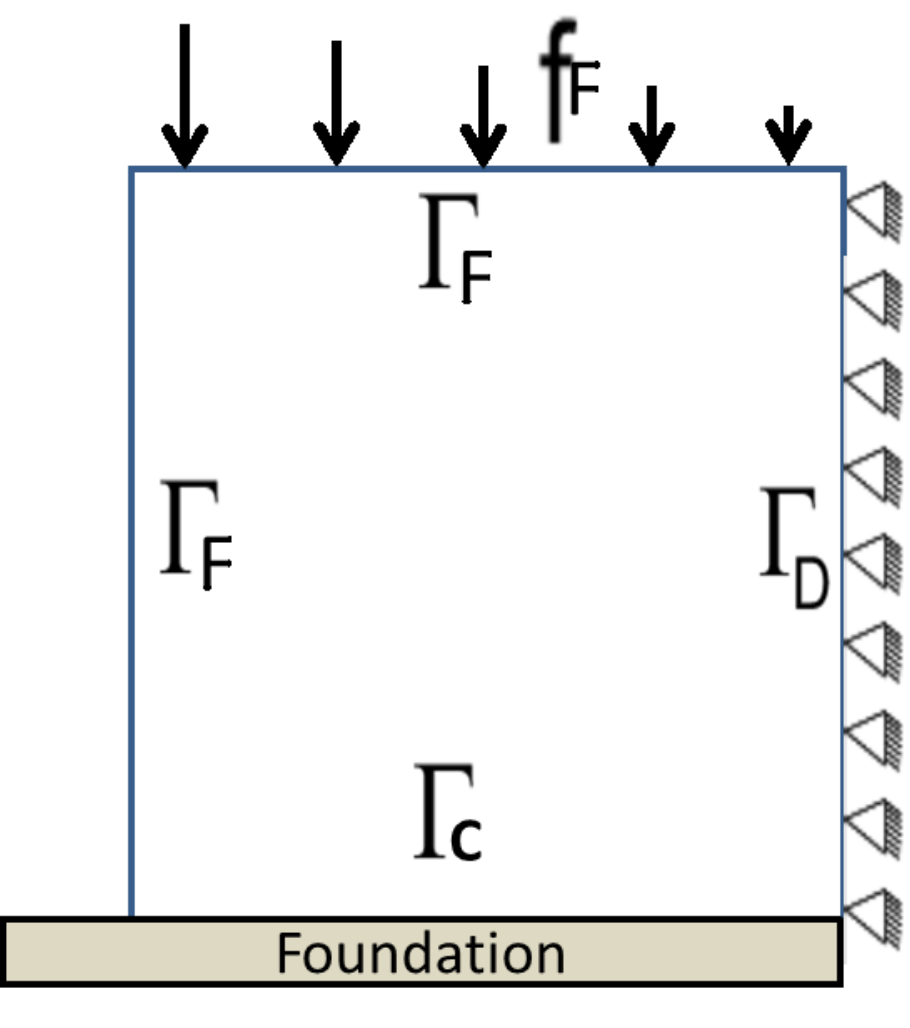}
\includegraphics[scale=0.3]{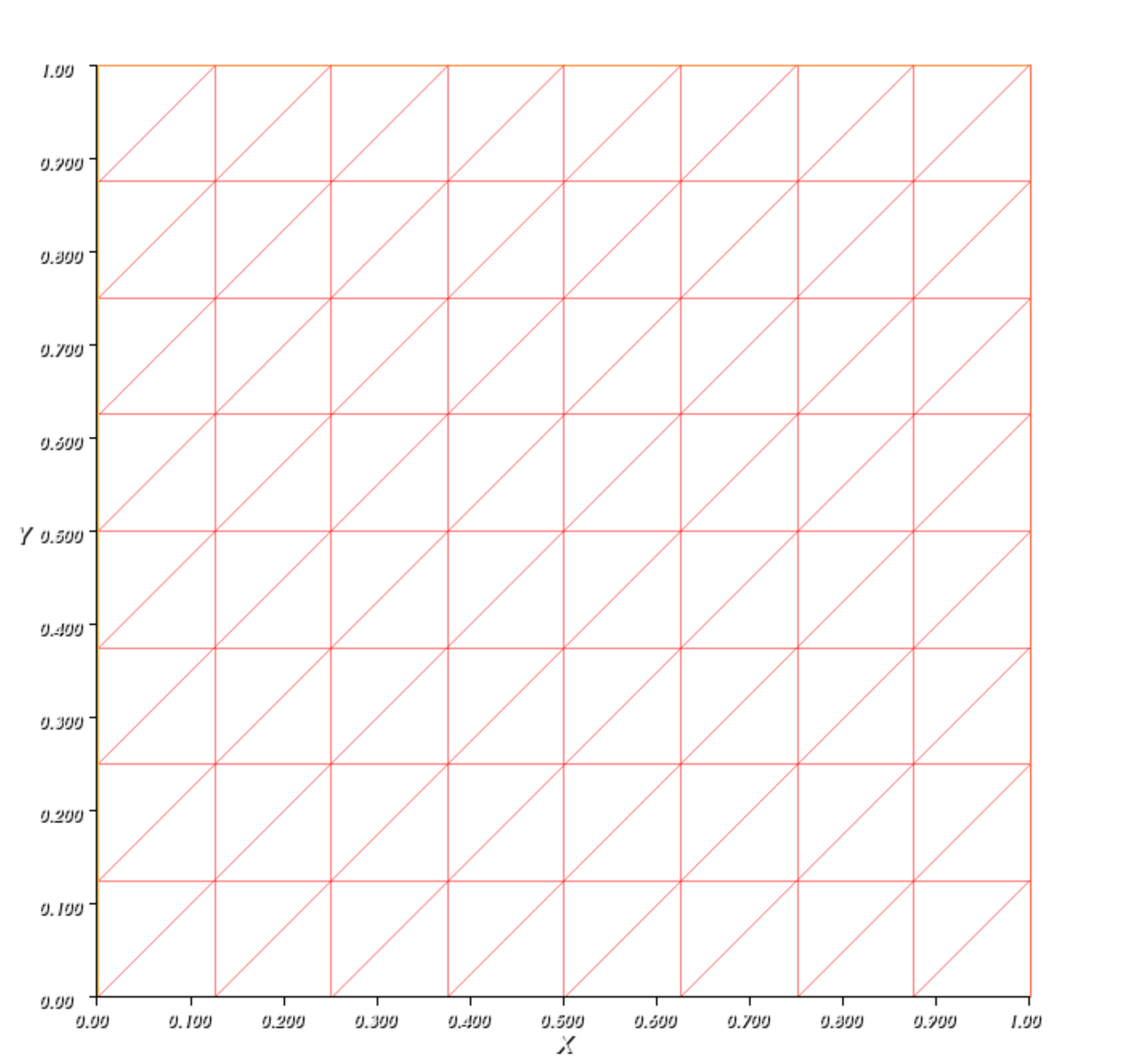}
\caption{Example 1: Physical setting and mesh example for $N_{el}$=8.}
\label{f-test1}
\end{center}
\end{figure}

\begin{table}
\begin{center}
\begin{tabular}{c c c c c }
\hline
& Piezoelectric $(C/m^2) $ & & Permittivity $(C^2/(Nm^2)$ & \\
\hline
$e_{21}$ & $e_{22}$  & $e_{13}$ & $\beta_{11}$ & $\beta_{22}$   \\
\hline
-5.4 & 15.8 & 12.3  & 916 & 830  \\
\hline
\end{tabular}
\caption{Material constants.} \label{piezct}
\end{center}
\end{table}

The numerical solution corresponding to $N_{el}=512$ subdivisions on
each outer side of the square (see the right-hand side of Fig.
\ref{f-test1} for the case $N_{el}=8$), and $k=0.00078612$ has been
considered as the ``exact'' solution in order to compute the
numerical errors given by
$$ E^{hk}=\displaystyle \max_{0\leq n\leq N} \left\{ \|\bu_n-\bu_n^{hk}\|_V+
\|v_n-v_n^{hk}\|_H+\|\varphi_n-\varphi_n^{hk}\|_W\right\}.  $$ Both
the piezoelectric and permittivity coefficients are depicted in
Table \ref{piezct}. Moreover the following data have been employed
in the simulations:
$$\begin{array}{l}
T= 1 \, s,\;\; \fb_0(\bx,t)=\bzero \; N/m^3,\;\; \fb_F(\bx,t)=\left\{\begin{array}{l}
(0,-60(1-x_1)t) \; N/m^2 \; \text{ if $x_2$=1,}\\
\bzero \;\: N/m^2 \quad \text{elsewhere,}
\end{array}\right. \\
E=20000 \; \; N/m^2, \quad  r=0.3, \quad c_p=10^5, \quad \rho=1 \; kg/m^3,\\
\varphi_A=0 \; \; V, \quad q_0=0\;\; C/m^3, \quad q_F=0 \; C/m^2,\\
\bu_0=\bzero \;\; m, \quad \bv_0=\bzero\;\; m/s, \quad
\varphi_0=0\;\; V.
\end{array}$$
In Table \ref{tt_el} the numerical errors obtained for some
discretization parameters $N_{el}$ and $k$ are shown. As can be
seen, the convergence of the numerical algorithm is clearly
observed.  The evolution of the error with respect to the parameter
$k+h$ is plotted in Figure \ref{f-recta} (here,
$h=\frac{\sqrt{2}}{N_{el}}$). The linear convergence of the
algorithm seems to be achieved.

\begin{table}
\begin{center}
\begin{tabular}{|c | c c c c c c c |}
\hline $N_{el} \downarrow k \to $ & 0.0015625 &  0.003125 &   0.00625 &   0.0125 &   0.025 &   0.05 &   0.1  \\
\hline
 4  &  0.470744  & 0.470809 &  0.470941 &  0.471220 &  0.471838 &  0.473329 &  0.477455 \\
  8  & 0.255173  & 0.255208  & 0.255284  & 0.255468  & 0.255957  & 0.257490  & 0.263109 \\
 16  & 0.141647  & 0.141660  & 0.141698 &  0.141839 &  0.142445 &  0.144762 &  0.152973 \\
 32 &  0.080098  & 0.080101 &  0.080149 &  0.080420 &  0.081371 &  0.084702 &  0.096829 \\
 64 &  0.045528 &  0.045547 &  0.045663  & 0.046057 &  0.047449 &  0.052511  & 0.069826 \\
128 &  0.025358 &  0.025405 &  0.025569 &  0.026165 &  0.028363  & 0.035951 &  0.058180 \\
256  & 0.012722  & 0.012791 &  0.013062 &  0.014099  & 0.017720 & 0.028208 &  0.053679 \\
\hline
\end{tabular}
\caption{Example 1: Numerical errors ($\times 100$) for some
$N_{el}$ and $k$.} \label{tt_el}
\end{center}
\end{table}

\begin{figure}[h]
\begin{center}
\includegraphics[width=7cm]{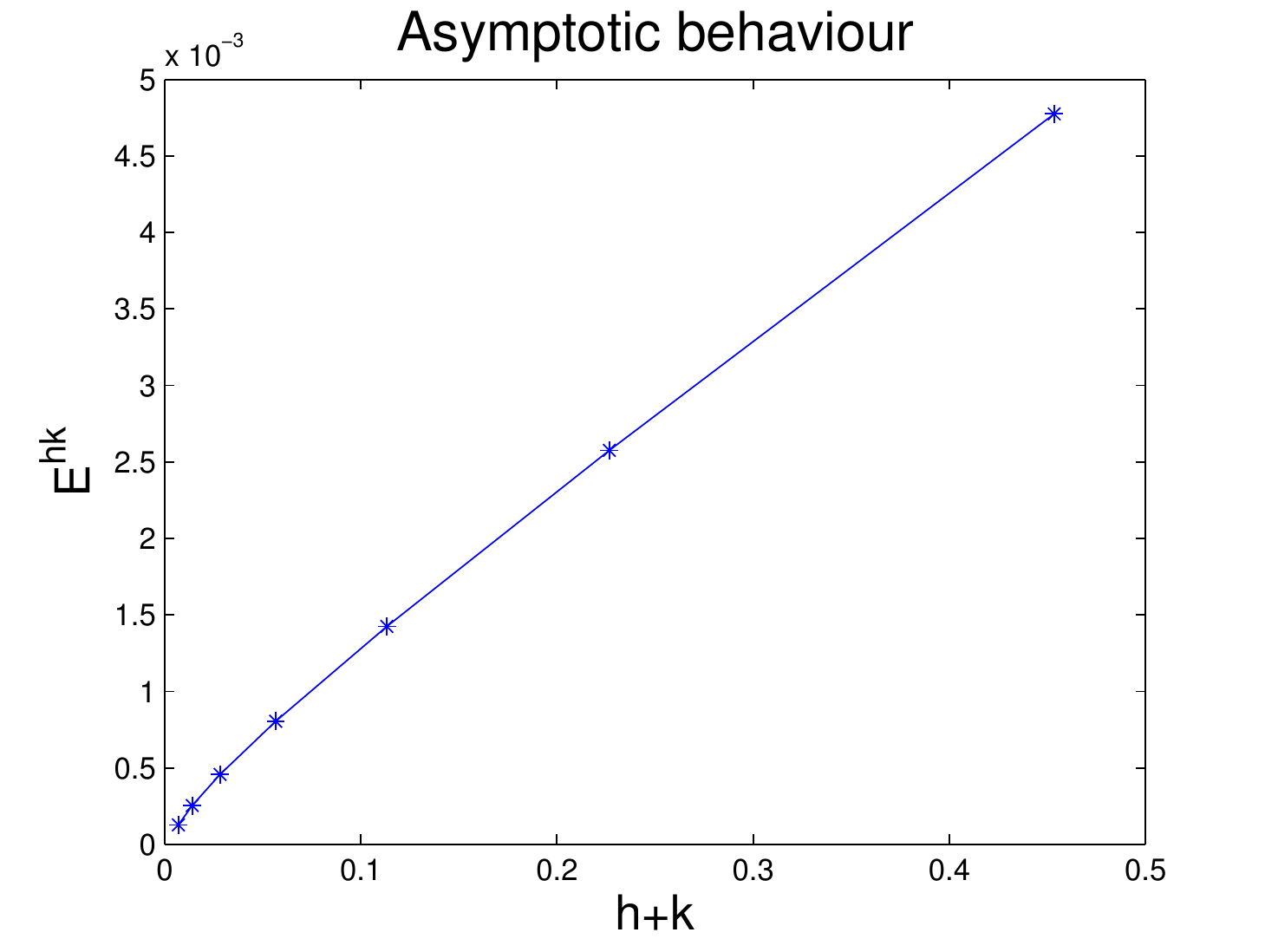}
\caption{Example 1: Asymptotic behaviour of the numerical scheme}
    \label{f-recta}
\end{center}
\end{figure}

\subsection{A second example: piezoelectric effect}

As a second numerical example, in order to observe the effect of the
piezoelectric properties of the material, a physical setting as the
one depicted in Fig. \ref{f-test2} is considered.
\begin{figure}[h]
\begin{center}
\includegraphics[scale=0.7]{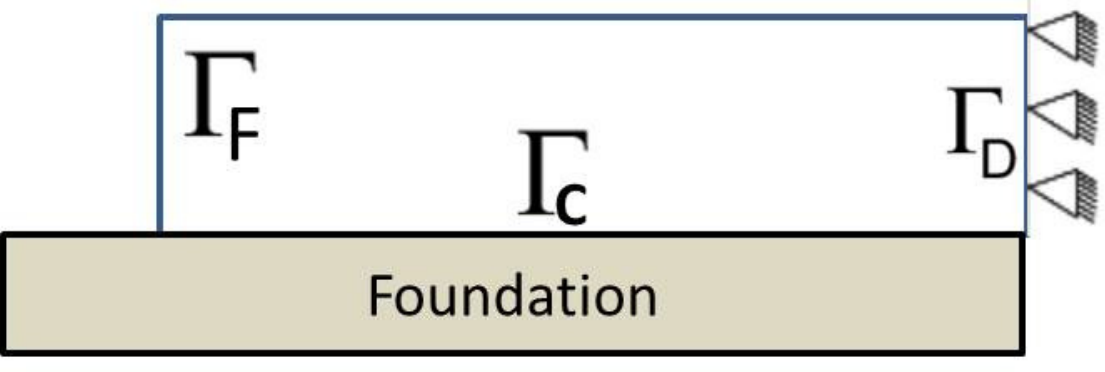}
\caption{Example 2: physical setting}
\label{f-test2}
\end{center}
\end{figure}

In this case the body $\Omega=[0,4]\times[0,1]$ is clamped on its
right end and it remains in initial contact with a deformable
foundation on its lower boundary. No physical forces act on the
body, but a constant surface  electric charge ($q_F=200\; C/M^2$) is
applied on the lower part of the boundary, where contact is
produced. Here, as in the previous example, we assume that
$\Gamma_D=\Gamma_A.$

The following data have been used
$$\begin{array}{l}
T= 1 \, s,\quad \fb_B(\bx,t)=\bzero \;\; N/m^3,\quad \fb_F=\bzero \;\: N/m^2,\\
E=2\times 10^6 \; \; N/m^2, \quad  r=0.3, \quad c_p=10^5,\quad \rho=1000 \; kg/m^3, \\
\varphi_A=0 \; \; V, \quad q_0=0\;\; C/m^3, \quad q_F=\left\{\begin{array}{l}
200\;\; C/m^3 \text{ if } x_2=0,\\
0 \quad \text{elsewhere},
\end{array}\right. \\
\bu_0=\bzero \;\; m, \quad \bv_0=\bzero \;\; m/s, \quad
\varphi_0=0\;\; V.
\end{array}$$

\begin{figure}[h]
\begin{center}
\includegraphics[scale=0.2]{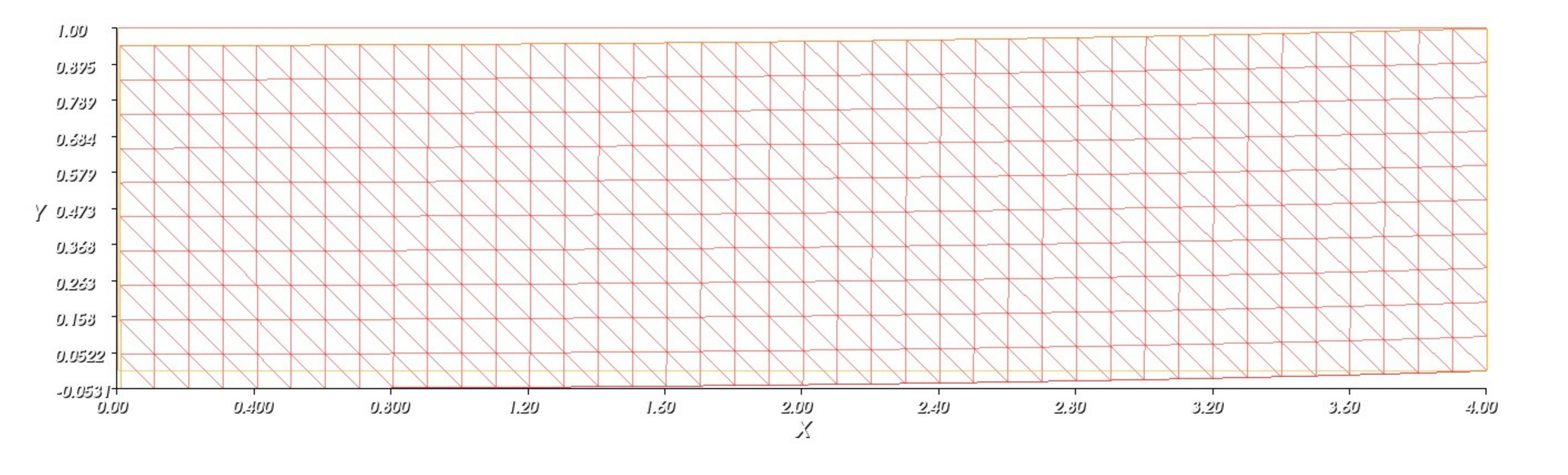}
\caption{Example 2: Deformed configuration (x 5000) at final time}
\label{f-def2}
\end{center}
\end{figure}

We can see  in Fig. \ref{f-def2} that deformations appear due to the
piezoelectric effect which, added to the mechanical restrictions,
lead the body to a stress-state which can be observed in Fig.
\ref{f-poten2} (right-hand side). In this figure (left-hand side),
the electric potential field  is shown at final time.
\begin{figure}[h]
\begin{center}
$\begin{array}{cc}
\includegraphics[scale=0.17]{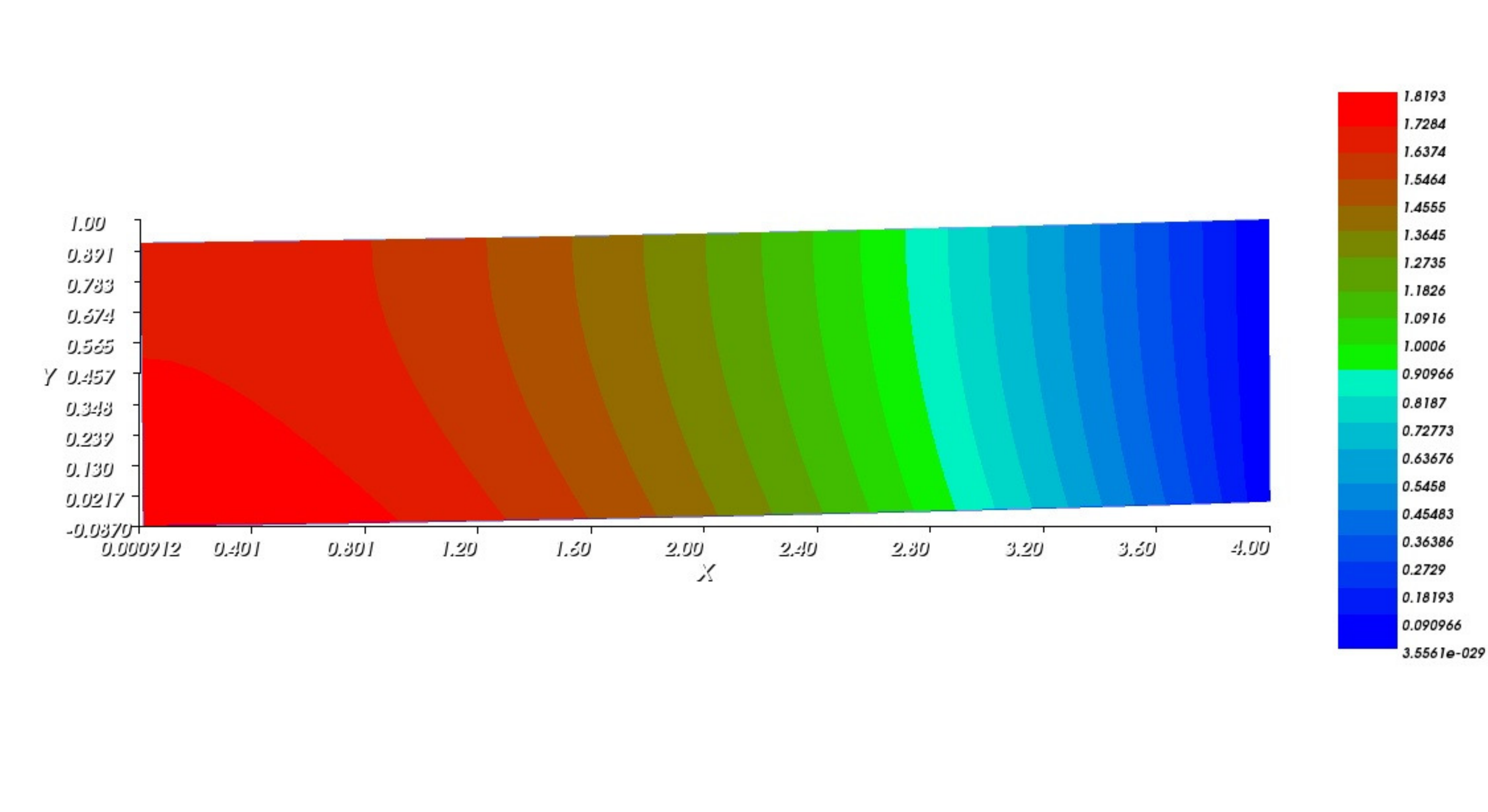}&
\includegraphics[scale=0.17]{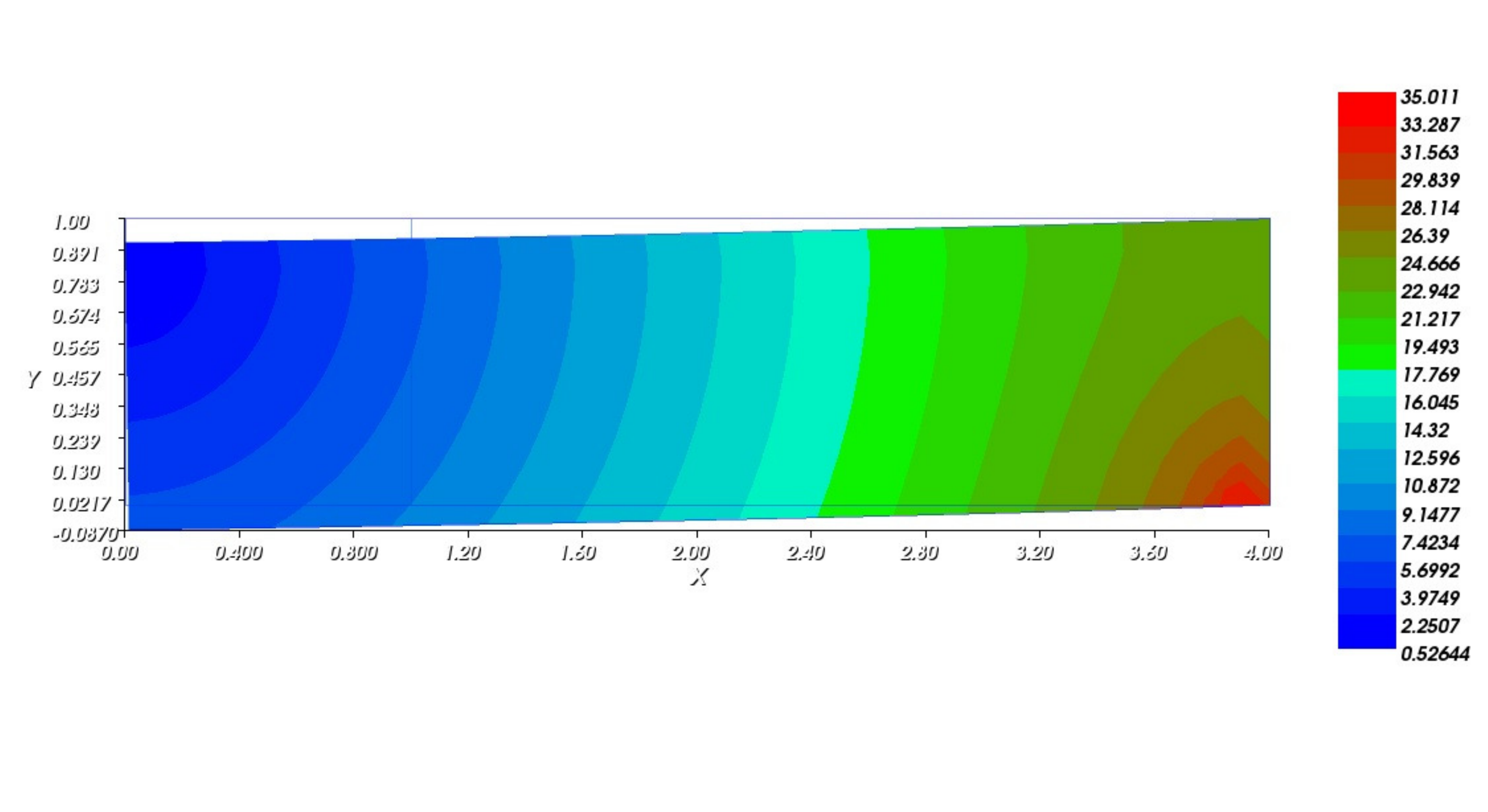}
\end{array}$
\caption{Example 2: potential field and von mises stress norm at final time.}
\label{f-poten2}
\end{center}
\end{figure}

\subsection{A third example: deformable contact of an L-shaped domain}

As a final example, we consider an L-shaped body which is submitted
to the action of traction forces on its upper horizontal boundary.
The body is clamped on its lower horizontally boundary and an
obstacle is assumed to be in initial contact, as can be observed in
Fig. \ref{f-def3}.

\begin{figure}[h]
\begin{center}
\includegraphics[scale=0.8]{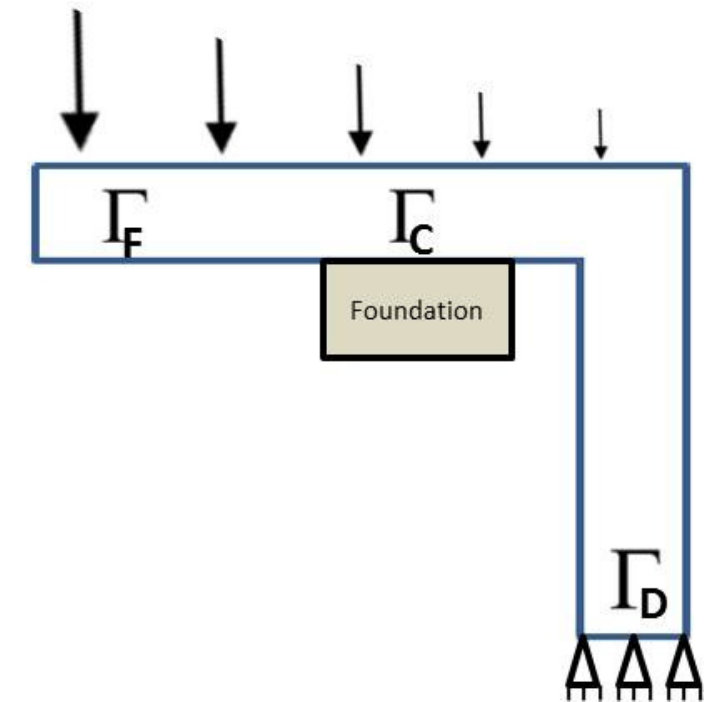}
\caption{Example 3: Contact problem of an L-shaped domain}
\label{f-def3}
\end{center}
\end{figure}

The following data are used in the simulation:
$$\begin{array}{l}
T= 1 \, s,\quad \fb_B(\bx,t)=\bzero \;\; N/m^3,\quad \fb_F=\left\{\begin{array}{l}
(0,-500(60-x_1)t) \; N/m^2 \; \text{ if $x_2$=50,}\\
\bzero \;\: N/m^2 \quad \text{elsewhere,}
\end{array}\right. \\
E=2.1\times 10^9 \; \; N/m^2, \quad  r=0.3, \quad c_p=10^5,\quad \rho=27000 \; kg/m^3, \\
\varphi_A=0 \; \; V, \quad q_0=0\;\; C/m^3, \quad q_F=0\;\; C/m^3, \\
\bu_0=\bzero \;\; m, \quad \bv_0=\bzero \;\; m/s, \quad
\varphi_0=0\;\; V.
\end{array}$$

Both electric potential and von Mises stress norm are plotted, over
the final configuration of the body and at final time, in Fig.
\ref{f-poten3}. The area of maximum stress concentration is located
near the contact boundary due to the bending movement, and it
coincides with the region where the electric potential reaches its
maximum value as it was expected.

\begin{figure}[h!]
\begin{center}
$\begin{array}{cc}
\includegraphics[scale=0.4]{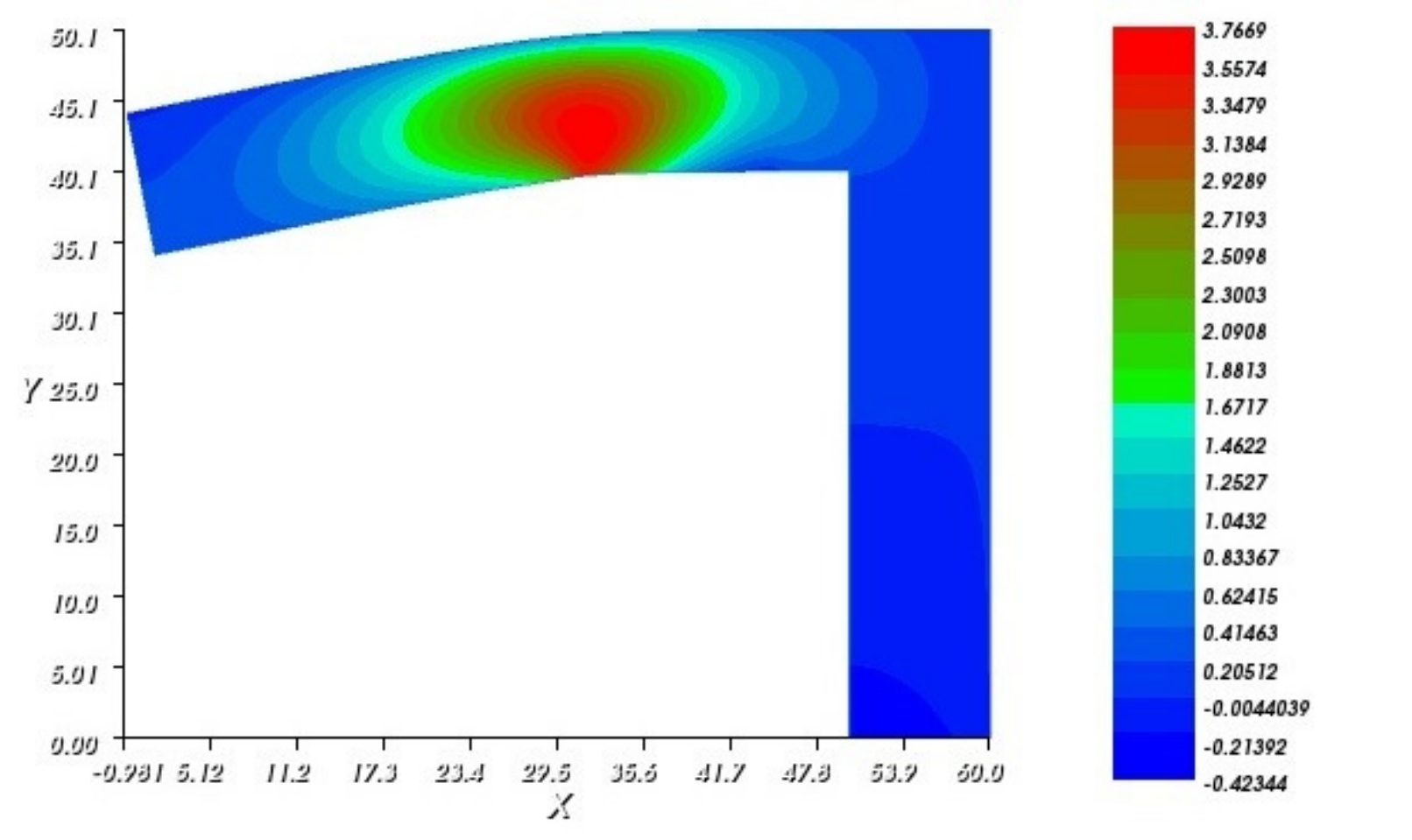} &
\includegraphics[scale=0.4]{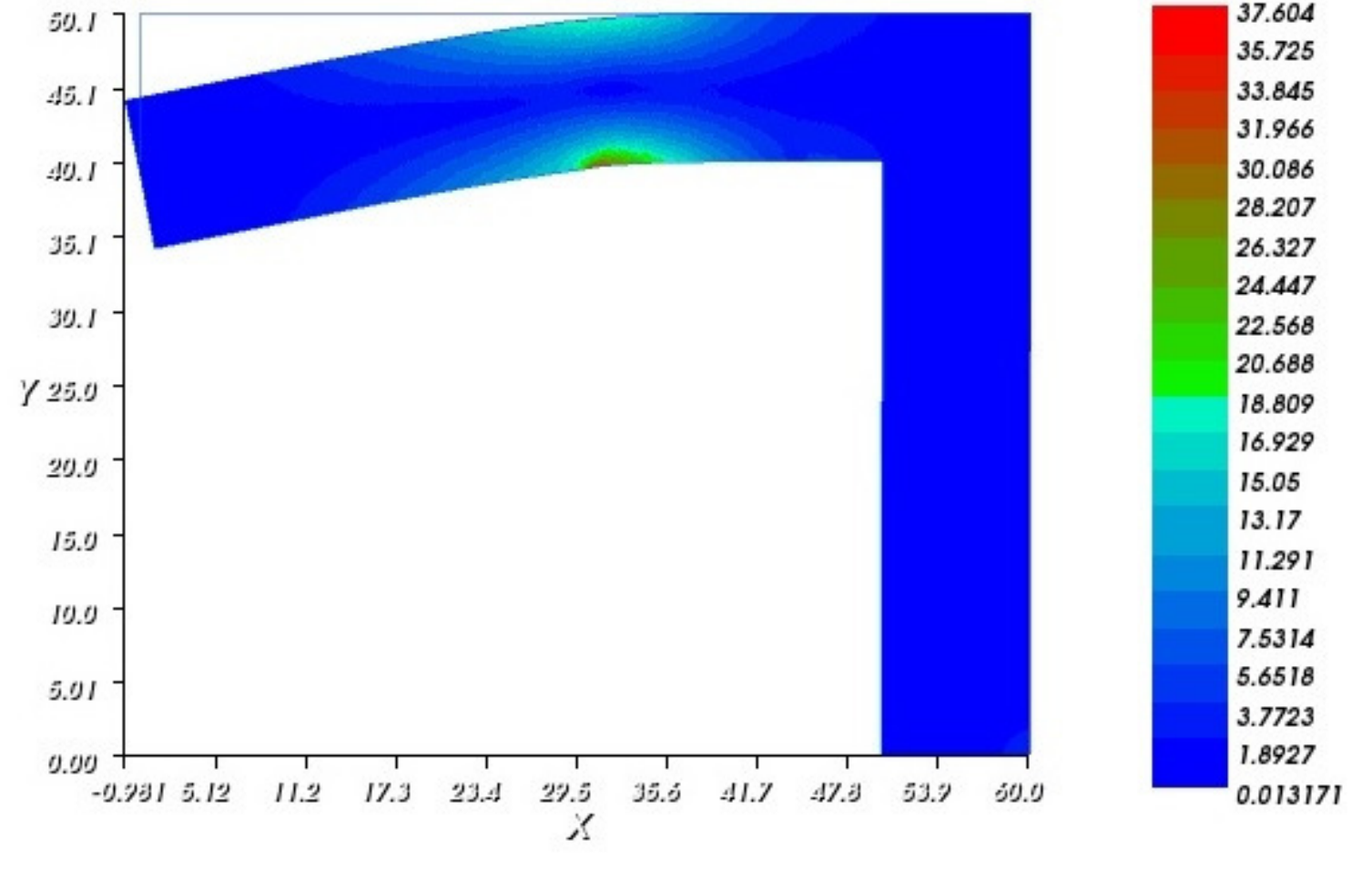}
\end{array}$
\caption{Example 3: potential field and von mises stress norm over the deformed mesh ($\times 500$) at final time.}
\label{f-poten3}
\end{center}
\end{figure}

\section*{Acknowledgements}

The work of M. Campo and J.R. Fern\'andez was supported by the
Ministerio de Econom\'ia y Competitividad under the research project
MTM2012-36452-C02-02 (with the participation of FEDER) and the work
of \'A. Rodr\'{\i}guez-Ar\'os and J.M. Rodr\'{\i}guez was supported
by the research project MTM2012-36452-C02-01 with the participation
of FEDER.

%------------------------------------------------------------------------------
% End of journal.tex
%------------------------------------------------------------------------------
\end{document}